\documentclass{article}

\usepackage{xspace} 

\usepackage{amsmath}
\usepackage{amsfonts}
\usepackage{amssymb}
\usepackage{dsfont} 
\usepackage{mathrsfs} 
\usepackage{mathtools} 
\usepackage{amsthm} 
\usepackage[dvipsnames]{xcolor} 
\allowdisplaybreaks 

\usepackage{hyperref}
\hypersetup{
    colorlinks,
    linkcolor={NavyBlue},
    citecolor={NavyBlue},
    urlcolor={black}}

\usepackage{geometry}
\renewcommand\labelenumi{(\roman{enumi})}
\renewcommand\theenumi\labelenumi

\makeatletter
\newcommand\footnoteref[1]{\protected@xdef\@thefnmark{\ref{#1}}\@footnotemark}
\makeatother

\newtheorem{theorem}{Theorem}[section]
\newtheorem{proposition}[theorem]{Proposition}
\newtheorem{lemma}[theorem]{Lemma}
\newtheorem{corollary}[theorem]{Corollary}
\newtheorem{conjecture}[theorem]{Conjecture}

\theoremstyle{definition}

\newtheorem{remark}[theorem]{Remark}

\newcommand{\E}{\mathbb{E}}

\newcommand{\N}{\mathbb{N}}

\renewcommand{\P}{\mathbb{P}}

\newcommand{\R}{\mathbb{R}}

\newcommand{\T}{\mathbb{T}}

\newcommand{\Z}{\mathbb{Z}}

\newcommand{\cD}{\mathcal{D}}

\newcommand{\cL}{\mathcal{L}}

\newcommand{\hG}{\mathscr{G}}
\newcommand{\hF}{\mathscr{F}}
\newcommand{\ie}{i.e.\@\xspace}
\newcommand{\eg}{e.g.\@\xspace}
\newcommand{\iid}{i.i.d.\@\xspace}
\renewcommand{\tilde}{\widetilde}
\renewcommand{\hat}{\widehat}
\renewcommand{\d}[1]{\mathop{}\!\mathrm{d}#1}
\newcommand{\e}{\mathrm{e}}
\newcommand{\vep}{\varepsilon}

\let\originalleft\left
\let\originalright\right
\renewcommand{\left}{\mathopen{}\mathclose\bgroup\originalleft}
\renewcommand{\right}{\aftergroup\egroup\originalright}

\title{Upper moderate deviation probabilities for the maximum of a branching random walk}
\author{Louis \textsc{Chataignier}%
\footnote{Institut de Math\'ematiques de Toulouse, UMR 5219, Universit\'e de Toulouse, CNRS, UPS, F-31062 Toulouse Cedex 9, France.}~
and Lianghui \textsc{Luo}$^*$}
\date{\today}

\begin{document}

\maketitle

\begin{abstract}
    Consider $M_n$ the maximal position at generation~$n$ of a supercritical branching random walk.
    In 2013, A\"id\'ekon~\cite{Aidekon2013} obtained and described the convergence in law, as time~$n$ goes to infinity, of $M_n-m_n$, where $m_n$ is an explicit function.
    Equivalently, he identified the limit of $\P(M_n > m_n + x)$, for any $x \in \R$.
    More recently, Luo~\cite{Luo2025} gave an asymptotic equivalent for the upper large deviation probability, that is $\P(M_n > m_n + xn)$, for $x > 0$.
    In this work, we study an intermediate regime, called \emph{upper moderate deviation}.
    We obtain, under close-to-optimal integrability conditions, an asymptotic equivalent for $\P(M_n > m_n + x_n)$, where~$x_n$ is such that $x_n \to \infty$ and $x_n = O(\sqrt{n})$.
    Our proof is based on a strategy due to Bramson, Ding, and Zeitouni~\cite{BramsonDingZeitouni2016}.
    As a byproduct, we obtain information about the typical behavior of particles contributing to such deviations.
    Finally, we apply our main result to show the convergence in law of the centered maximum of a two-speed branching random walk in the mean regime and describe its limit.
\end{abstract}

\tableofcontents

\section{Introduction}

The branching random walk is a natural extension of the Galton--Watson tree, with the addition of a spatial attribute for each particle.
In one dimension, it is constructed as follows.
At time~$n=0$, a single particle is located at the origin of the real line.
At time~$n=1$, this particle dies and gives birth to children, scattered on~$\R$ according to some point process~$\cL$.
These new particles then restart the process, independently of each other, scattering children according to independent copies of~$\cL$, translated by their respective positions.
The system continues indefinitely, with possible extinction if there is no particle at some generation.

Given a particle~$u$, we denote by $|u|$ its generation and by $V(u)$ its position.
Our object of interest is the maximal displacement $M_n = \max_{|u|=n} V(u)$.
We work under standard assumptions for the study of its asymptotic behavior.
First, we assume the underlying Galton--Watson process to be supercritical, \ie
\begin{equation}\label{eq:assumption_supercriticality}
    \E \left[ \sum_{|u|=1} 1 \right] \in (1, \infty].
\end{equation}
Thus, the particle population survives indefinitely with positive probability (see \eg~\cite{Harris1963}).
Let us mention that we allow~$\cL$ to have infinitely many atoms with positive probability.
Moreover, we work in the setting where
\begin{equation}\label{eq:assumption_boundary_case}
	\E \left[ \sum_{|u|=1} \e^{V(u)} \right] = 1 \quad \text{and} \quad \E \left[ \sum_{|u|=1} V(u)\e^{V(u)} \right] = 0,
\end{equation}
known in the literature as the \emph{boundary case}\footnote{Actually, most authors consider the \emph{boundary case} in the sense of~\cite{BigginsKyprianou2005}, \ie the setting where~\eqref{eq:assumption_boundary_case} holds when replacing the positions $V(u)$ with their opposite $-V(u)$.
Then, they are interested in the minimal position instead of $M_n$.}.
See appendices of~\cite{Jaffuel2009} or~\cite{BerardGouere2011} for a discussion on the (rare) cases where a branching random walk cannot be reduced to Assumption~\eqref{eq:assumption_boundary_case}.
Writing for $x\in\R\cup\{\pm\infty\}$, $x_+\coloneq\max(x,0)$ and $x_-\coloneq\max(-x,0)$ and $\log_+x\coloneq\log(\max(1,x))$, we assume the following integrability conditions:
\begin{equation}\label{eq:assumption_gaussianity}
    \sigma^2 = \E\left[\sum_{|u|=1}V(u)^2 \e^{V(u)}\right] < \infty,
\end{equation}
and
\begin{equation}\label{eq:assumption_peeling_lemma}
   \E[X(\log_+X)^2] < \infty, \quad \E[\tilde{X}\log_+ \tilde{X}] < \infty,
\end{equation}
where 
\[X\coloneq\sum_{|u|=1} \e^{V(u)}, \quad \tilde{X}\coloneq\sum_{|u|=1} V(u)_- \e^{V(u)}.\]
We say that the branching random walk is $(h,a)$-\emph{lattice} if $0 \le a < h$ and~$h$ is the largest number such that the support of~$\cL$ is contained in $a+h\Z$.
In this case,~$h$ is called the \emph{lattice span}.
If such~$a$ and~$h$ do not exist, then we say that the branching random walk is \emph{non-lattice}.
In this paper, we work in both lattice and non-lattice cases, but it is sometimes necessary to treat them separately.

Before stating our main result, let us recall some known properties that hold under Assumptions~\eqref{eq:assumption_supercriticality}, \eqref{eq:assumption_boundary_case}, \eqref{eq:assumption_gaussianity}, and~\eqref{eq:assumption_peeling_lemma}.
For $n \ge 1$, set
\begin{equation*}
    m_n = - \frac{3}{2} \log n,
\end{equation*}
and define the \emph{derivative martingale}
\begin{equation}\label{eq:derivative_martingale}
    Z_n = -\sum_{|u|=n} V(u) \e^{V(u)}.
\end{equation}
Biggins and Kyprianou~\cite{BigginsKyprianou2004} and A\"id\'ekon~\cite{Aidekon2013} proved that~$Z_n$ converges almost surely as $n \to \infty$ and that the limit~$Z_\infty$ is almost surely positive on the non-extinction event.
Chen~\cite{Chen2015b} showed that~\eqref{eq:assumption_peeling_lemma} is a necessary assumption for $Z_\infty$ to be non-trivial.
In the non-lattice case, A\"id\'ekon~\cite{Aidekon2013} proved that there exists a constant~$C_{\mathrm{NL}}>0$ such that, for any $x \in \R$,
\begin{equation}\label{eq:aidekon}
    \lim_{n \to \infty} \P(M_n \le m_n + x) = \E \left[ \e^{-C_{\mathrm{NL}} Z_\infty\e^{-x}} \right].
\end{equation}
In other words, the centered maximum converges in law, conditionally on the non-extinction event, to a standard Gumbel variable with an independent random shift $\log(C_{\mathrm{NL}}Z_{\infty})$.
In the lattice case, there is no convergence because of the logarithm in~$m_n$, but tightness still holds, by Chen~\cite[Equation~4.20]{Chen2015a} or Mallein~\cite[Theorem~1.1]{Mallein2016}.
In both lattice and non-lattice cases, sharp upper and lower bounds were obtained for the tail of $M_n-m_n$, by Hu~\cite{Hu2016} and Mallein~\cite{Mallein2016} (see~\eqref{eq:hu_upper_bound} and~\eqref{eq:mallein_lower_bound} below).

Let us also mention that, under stronger assumptions than~\eqref{eq:assumption_supercriticality}, \eqref{eq:assumption_boundary_case}, \eqref{eq:assumption_gaussianity}, and~\eqref{eq:assumption_peeling_lemma}, Bramson, Ding and Zeitouni~\cite{BramsonDingZeitouni2016} explained how to adapt their arguments for the non-lattice case to obtain the analog of~\eqref{eq:aidekon} for the $(h,a)$-lattice case.
More precisely, there exists a constant~$C_{\mathrm{L}}>0$ such that
\begin{equation}\label{eq:bramson_ding_zeitouni}
    \lim_{n \to \infty} \sup_{x \in A_n} \left| \P(M_n \le m_n + x) - \E \left[ \e^{-C_{\mathrm{L}} Z_\infty\e^{-x}} \right] \right| = 0,
\end{equation}
where $A_n = \{x \in \R : (x+m_n-an)/h \in \Z\}$.

More recently, Luo~\cite{Luo2025} obtained an asymptotic equivalent for the upper large deviation probability of maximum: under some integrability conditions, for $x>0$, there exist $\kappa^*(x),C(x)>0$ such that 
\[\lim_{n\to\infty}\sqrt{n}\e^{\kappa^*(x)n}\P(M_n>xn)=C(x),\]
which strengthens previous results from Gantert and H\"ofelsauer~\cite{GantertHofelsauer2018}.

In this paper, we study an intermediate regime between typical fluctuations and upper large deviations, called \emph{upper moderate deviation}.
This concerns the event where $M_n > m_n + x_n$, with~$x_n$ going to infinity while $x_n/n$ goes to $0$.
In the framework of the binary branching Brownian motion, a similar spatial branching process in continuous time~$t \in [0,\infty)$, upper moderate deviations have already been studied.
Asymptotic equivalents were first obtained in specific regimes using estimates for solutions of the F--KPP equation~\cite{Bramson1983}: when $x_t$ grows at logarithmic rate~\cite[Theorem~1]{ChauvinRouault1990}, when $x_t$ is equivalent to~$\sqrt{t}$~\cite[Lemma~4.5]{ArguinBovierKistler2013}, and more recently, when $x_t$ goes to infinity while $(x_t \log t)/t$ goes to~$0$~\cite[Proposition~2.1]{BovierHartung2020}\footnote{Note that~\cite[Proposition~2.1]{BovierHartung2020} is however incorrectly stated for the regime where $x_t > 0$ and $x_t/t \to \infty$ (see~\cite{Chataignier2025}).}. Finally, Chataignier~\cite{Chataignier2025} treated the general regime where $x_t$ goes to infinity while $x_t/t$ goes to~$0$.
In the setting analogous to~\eqref{eq:assumption_boundary_case} and~\eqref{eq:assumption_gaussianity} for branching Brownian motion (see \eg~\cite{AidekonBerestyckiBrunetShi2013}), his main result reads
\begin{equation}\label{eq:moderate_deviation_bbm}
     \lim_{t \to \infty} \frac{\P(M_t > m_t + x_t)}{x_t \e^{-x_t - (x_t - (3/2)\log t)^2/(2t\sigma^2)}} = C_{\mathrm{BBM}},
\end{equation}
where 
\[C_{\mathrm{BBM}} \coloneq \frac{1}{\sigma^3} \sqrt{\frac{2}{\pi}} \lim_{s \to \infty} \int_0^\infty y \e^y \P \left( M_s > y \right) \d{y} \in (0,\infty).\]

Going back to branching random walk, an asymptotic equivalent of $\P(M_n>m_n+x_n)$ was already given by A\"id\'ekon~\cite[Proposition~4.1]{Aidekon2013} in the regime where $x_n\to\infty$ and $(3/2)\log n - x_n \to \infty$ as $n\to\infty$.
Hu~\cite{Hu2016} obtained asymptotic estimates for the \emph{lower} deviation probabilities $\P(M_n \le m_n - x_n)$ if~$x_n$ is negligible compared to $\log n$.
Under additional assumptions, Chen and He~\cite{ChenHe2020} extended these estimates to the case where~$x_n$ is at most of order~$n$.

The asymptotic order for the upper moderate deviation probabilities is already known thanks to Hu~\cite{Hu2016} and Mallein~\cite{Mallein2016}.
Hu obtained the following upper bound in~\cite[Lemma~4.4]{Hu2016}:
there exists $C>0$ such that, for any $n \ge 1$ and $y \ge 0$,
\begin{equation}\label{eq:hu_upper_bound}
    \P(M_n > m_n+y) \le C (1+y)\e^{-y}.
\end{equation}
Mallein obtained the following lower bound in~\cite[Theorem~4.1]{Mallein2016}:
for any $A > 0$, there exists $c>0$ such that, for any $n \ge 1$ and $y \in [0, A\sqrt{n}]$,
\begin{equation}\label{eq:mallein_lower_bound}
    \P(M_n > m_n+y) \ge c (1+y)\e^{-y}.
\end{equation}
Our main result Theorem~\ref{th:moderate_deviation} yields an asymptotic equivalent for $\P(M_n > m_n + x_n)$ when~$x_n$ goes to infinity at a rate at most of order~$\sqrt{n}$.
Thus, it is a refinement of~\eqref{eq:hu_upper_bound} and~\eqref{eq:mallein_lower_bound} in this regime.

As an application of our main result, we show in Theorem~\ref{th:two_speed} the analog of~\eqref{eq:aidekon} and~\eqref{eq:bramson_ding_zeitouni} for a two-speed branching random walk in the so-called \emph{mean regime}, which was conjectured in~\cite{Luo2025+} for the non-lattice case.
More precisely, we assume that particles reproduce according to different reproduction laws depending on whether their generation is before or after~$tn$, for some $t \in (0,1)$, and that both point processes satisfy~\eqref{eq:assumption_supercriticality}, \eqref{eq:assumption_boundary_case}, \eqref{eq:assumption_gaussianity}, and~\eqref{eq:assumption_peeling_lemma}.
As in the time-homogeneous branching random walk, in the non-lattice case, the centered maximum converges in law to a standard Gumbel variable with an independent random shift.
Regarding the lattice case of the time-homogeneous branching random walk, our theorem implies~\eqref{eq:bramson_ding_zeitouni}, which is in itself a new result under Assumptions~\eqref{eq:assumption_supercriticality}, \eqref{eq:assumption_boundary_case}, \eqref{eq:assumption_gaussianity}, and~\eqref{eq:assumption_peeling_lemma}.

Throughout this paper, the letters~$C$ and~$c$ will denote constants, \ie deterministic and time-invariant quantities, which may vary from one line to another and depend only on the law of the branching random walk.
For $x, y \in \R$, we will write $x \wedge y = \min(x,y)$ and $x \vee y = \max(x,y)$.
We will write $\lfloor x \rfloor$ the integer part of~$x$, \ie the greatest number in~$\Z$ such that $\lfloor x \rfloor \le x$.
Given a real sequence $(x_n)_{n \ge 0}$ and a positive sequence $(y_n)_{n \ge 0}$, we will write either $x_n = o_{n \to \infty}(y_n)$, $x_n = o_n(y_n)$, or $x_n = o(y_n)$ if $\lim_{n \to \infty} x_n/y_n = 0$, $x_n = O(y_n)$ if $\limsup_{n \to \infty} |x_n|/y_n < \infty$, $x_n \sim y_n$ if $\lim_{n \to \infty} x_n/y_n = 1$, and $x_n \asymp y_n$ if $0 < \liminf_{n \to \infty} x_n/y_n \le \limsup_{n \to \infty} x_n/y_n < \infty$.
Given $x \in \R$ and $A \subset \R$, we will write $x \pm A \coloneq \{x \pm a : a \in A\}$ and $xA \coloneq \{xa : a \in A\}$.
Finally, we work with the convention $\sup \varnothing = -\infty$ and $\inf \varnothing = +\infty$.

\subsection{Main result}

Here is the main result of this paper.

\begin{theorem}\label{th:moderate_deviation}
    Assume~\eqref{eq:assumption_supercriticality}, \eqref{eq:assumption_boundary_case}, \eqref{eq:assumption_gaussianity}, \eqref{eq:assumption_peeling_lemma}, and let $(x_n)$ be any sequence such that $x_n \to \infty$ and $x_n = O(\sqrt{n})$ as $n \to \infty$.
    \begin{enumerate}
        \item If the branching random walk is non-lattice, then
        \begin{equation}\label{eq:moderate_deviation_nl}
            \P(M_n > m_n + x_n) \sim C_{\mathrm{NL}} x_n \e^{-x_n - x_n^2/(2n\sigma^2)},
        \end{equation}
        where
        \begin{equation}\label{eq:C_NL}
            C_{\mathrm{NL}} \coloneq \frac{1}{\sigma^3} \sqrt{\frac{2}{\pi}} \lim_{\ell\to\infty} \int_0^\infty y \e^{y} \P(M_\ell>y) \d{y} \in (0,\infty).
        \end{equation}
        \item If the branching random walk is $(h,a)$-lattice and $x_n + m_n \in an + h\Z$ for each~$n$, then
        \begin{equation}\label{eq:moderate_deviation_l}
            \P(M_n > m_n + x_n) \sim C_{\mathrm{L}} x_n \e^{-x_n - x_n^2/(2n\sigma^2)},
        \end{equation}
    where
    \begin{equation}\label{eq:C_L}
        C_{\mathrm{L}} \coloneq \frac{h}{\sigma^3}\sqrt{\frac{2}{\pi}}\lim_{\ell\to\infty}\sum_{y\in(a\ell+h\Z)\cap[0,\infty)}y\e^y\P(M_\ell>y)\in(0,\infty).
    \end{equation}
    \end{enumerate}
\end{theorem}

\begin{remark}
    The constant~$C_{\mathrm{NL}}$ is also the one appearing in A\"id\'ekon's main result~\cite[Theorem~1.1]{Aidekon2013}, in view of Corollary~\ref{cor:tail_distribution} below and~\cite[Proposition~1.3]{Aidekon2013}.
    This constant actually appears in a lot of papers on both branching random walks and branching Brownian motion, \eg in~\cite{BramsonDingZeitouni2016,ArguinBovierKistler2013,BovierHartung2020,Chataignier2025} with the formulation~\eqref{eq:C_NL} and in~\cite{Madaule2017,LalleySellke1987,ChauvinRouault1990,AidekonBerestyckiBrunetShi2013} without such a formulation.
\end{remark}

\begin{remark}\label{rem:unification}
    We will see in~\eqref{eq:unification_limsup} that~\eqref{eq:moderate_deviation_nl} and~\eqref{eq:moderate_deviation_l} can be unified in the following way.
    Let~$h$ denote the lattice span in the lattice case, and $h=0$ in the non-lattice case.
    Then,
    \begin{equation*}
            \P(M_n > m_n + x_n) \sim \rho(h) C^* x_n \e^{-x_n - x_n^2/(2n\sigma^2)},
        \end{equation*}
        where
    \begin{equation}\label{eq:unification}
        \rho(h) = \begin{cases}
            \frac{h}{\e^h - 1} & \text{if } h > 0, \\
            1 & \text{if } h = 0,
        \end{cases} \quad \text{and} \quad C^* = \frac{1}{\sigma^3} \sqrt{\frac{2}{\pi}} \lim_{\ell \to \infty} \int_0^\infty y \e^y \P(M_\ell>y) \d{y} \in (0,\infty).
    \end{equation}
\end{remark}

Since Theorem~\ref{th:moderate_deviation} holds for any sequence $(x_n)$ going to infinity arbitrarily slowly, an argument by contradiction (see Lemma~\ref{lem:from_sequence_to_large_constant}) yields the following asymptotic for the tail distribution of $M_n$.

\begin{corollary}\label{cor:tail_distribution}
    Take the same assumptions as in Theorem~\ref{th:moderate_deviation}.
    \begin{enumerate}
        \item If the branching random walk in non-lattice, then
        \begin{equation*}
            \limsup_{x \to \infty} \limsup_{n \to \infty} \left| \frac{\P(M_n > m_n + x)}{C_{\mathrm{NL}} x \e^{-x}} - 1 \right| = 0.
        \end{equation*}
        \item If the branching random walk in $(h,a)$-lattice, then
        \begin{equation*}
            \limsup_{x \to \infty} \limsup_{n \to \infty} \left| \frac{\P(M_n > m_n + x)}{C_{\mathrm{L}} n^{-3/2} x \e^{-[m_n+x]_n}} - 1 \right| = 0,
        \end{equation*}
        where $[y]_n$ is the greatest number in $an+h\Z$ such that $[y]_n \le y$, \ie $[y]_n = an + h \lfloor (y-an)/h \rfloor$.
    \end{enumerate}
\end{corollary}

In the non-lattice case, the above asymptotic was already present in the works of A\"id\'ekon~\cite[Proposition~1.3]{Aidekon2013} and Bramson, Ding, and Zeitouni~\cite[Proposition~3.1]{BramsonDingZeitouni2016}.
However, the former did not provide the formulation~\eqref{eq:C_NL} for the constant~$C_{\mathrm{NL}}$, while the latter did, but worked under stronger assumptions.

Since Theorem~\ref{th:moderate_deviation} is stated for general choices of the sequence $(x_n)$, a similar argument to that of Corollary~\ref{cor:tail_distribution} (see Lemma~\ref{lem:from_sequence_to_uniformity}) shows that the asymptotic equivalent also holds uniformly for~$x_n$, in the following sense.

\begin{corollary}\label{cor:uniform_convergence}
    Take the same assumptions, the same notation as in Theorem~\ref{th:moderate_deviation}, and fix two sequences $(a_n)$ and $(b_n)$ such that $a_n \le b_n$, $a_n \to \infty$, and $b_n = O(\sqrt{n})$ as $n \to \infty$.
    Then,
    \begin{equation*}
        \limsup_{n \to \infty} \sup_{x \in A_n} \left| \frac{\P(M_n > m_n + x)}{C_\circ x \e^{-x - x^2/(2n\sigma^2)}} - 1 \right| = 0,
    \end{equation*}
    where
    \begin{enumerate}
        \item in the non-lattice case, $C_\circ = C_{\mathrm{NL}}$ and $A_n = [a_n,b_n]$,
        \item in the $(h,a)$-lattice case, $C_\circ = C_{\mathrm{L}}$ and $A_n = [a_n,b_n] \cap \{x : m_n + x \in an + h\Z\}$.
    \end{enumerate}
\end{corollary}

Now, let us formulate two conjectures.
Note that, if the branching random walk is non-lattice and if we set $f(y) = \mathds{1}_{\{y \le 0\}}$, then Theorem~\ref{th:moderate_deviation} states that
\begin{equation*}
    \E \left[ 1 - \prod_{|u|=n} f(V(u)-m_n-x_n) \right] \sim C_{\mathrm{NL}} x_n \e^{-x_n - x_n^2/(2n\sigma^2)}.
\end{equation*}
Considering $\phi \coloneq -\log f$, this makes a link with the \emph{extremal process}, \ie the point process formed by the branching random walk seen from $m_n$, and suggests the following statement.

\begin{conjecture}\label{conj:laplace_estimate}
    Assume~\eqref{eq:assumption_supercriticality}, \eqref{eq:assumption_boundary_case}, \eqref{eq:assumption_gaussianity}, \eqref{eq:assumption_peeling_lemma}, and the branching random walk to be non-lattice.
    Let $\phi \colon \R \to \R$ be continuous with compact support and $(x_n)$ be a sequence such that $x_n \to \infty$ and $x_n = O(\sqrt{n})$ as $n \to \infty$.
    Then, as $n \to \infty$,
    \begin{equation*}
        \E \left[ 1 - \e^{-\sum_{|u|=n} \phi(V(u)-m_n-x_n)} \right] \sim C_{\mathrm{NL}}(\phi,a) x_n \e^{-x_n - x_n^2/(2n\sigma^2)},
    \end{equation*}
    where
    \begin{equation*}
        C_{\mathrm{NL}}(\phi) \coloneq \frac{1}{\sigma^3}\sqrt{\frac{2}{\pi}}\lim_{\ell\to\infty}\int_0^\infty y \e^{y} \E \left[ 1 - \e^{-\sum_{|u|=\ell} \phi(V(u)-y)} \right] \d{y}.
    \end{equation*}
\end{conjecture}

The extremal process has been studied by Madaule~\cite{Madaule2017}.
He showed that, conditionally on non-extinction, with respect to the vague topology, it converges in law to a decorated Poisson point process, \ie a Poisson point process to which we attach, at each atom~$x$, a point process~$\cD^{(x)}$ shifted by~$x$, where the~$\cD^{(x)}$ are independent copies of some point process~$\cD$, called \emph{decoration}.
This result is the analog of the main results from~\cite{ArguinBovierKistler2013,AidekonBerestyckiBrunetShi2013} for branching Brownian motion.
However, contrary to the latter, no characterization of the law of~$\cD$ has been established so far for branching random walks.

We believe that Theorem~\ref{th:moderate_deviation} and Conjecture~\ref{conj:laplace_estimate} will allow us to adapt the proof of Arguin, Bovier, and Kistler~\cite{ArguinBovierKistler2013} to branching random walks.
Our statements would play the same role as sharp estimates for solutions of the F--KPP equation, used \eg to prove~\cite[Theorem~3.4]{ArguinBovierKistler2013}.
One consequence would be a characterization of the decoration~$\cD$ as a branching random walk conditioned on a moderate deviation of its maximum.

\subsection{Strategy of the proof}

Let us discuss our strategy for the proof of Theorem~\ref{th:moderate_deviation}.
A classical technique for studying branching processes is to compute expectations of certain functionals thanks to the many-to-one formula or the spinal decomposition (introduced in Section~\ref{sct:spinal_decomposition} below).
It allows to reduce the computation to an expectation involving only one random walk.
A naive idea for the upper bound in Theorem~\ref{th:moderate_deviation} would be to apply the union bound
\begin{equation*}
    \P(M_n > m_n+x_n) \le \E \left[ \sum_{|u|=n} \mathds{1}_{\{V(u) > m_n+x_n\}} \right],
\end{equation*}
and then the many-to-one formula.
This actually yields the correct asymptotic order, up to polynomial factors.
The expectation turns out to be inflated by an atypical event (more atypical than $\{M_n > m_n+x_n\}$) where some particles reach really high positions before time~$n$.
As done by Mallein~\cite{Mallein2016}, we can truncate the process by adding a suitable barrier: only count particles that are above $m_n+x_n$ at time~$n$ while having stayed below $f_n(k)+x_n = m_n-m_{n-k}+x_n$ for all times $k<n$.
This yields the correct asymptotic order.

To obtain an asymptotic equivalent, we need a refinement of this truncation method.
Such a refinement was developed by Bramson, Ding, and Zeitouni~\cite{BramsonDingZeitouni2016} to prove the convergence of $M_n-m_n$.
In this paper, we adapt their approach in the same way as~\cite{Chataignier2025}, but with additional technical difficulties due to our close-to-optimal assumptions.
We consider an intermediate time~$\ell = \ell_n$ such that $\ell \to \infty$ as $n \to \infty$.
We denote by~$\Lambda_{n,\ell}$ the number of particles alive at time~$n-\ell$ whose ancestors have stayed below $f_{n-\ell}(k)+x_n$ for all times $k \le n-\ell$ and have a descendant reaching $m_n+x_n$ at time~$n$.
Using the upper and lower bounds~\eqref{eq:hu_upper_bound} and~\eqref{eq:mallein_lower_bound}, we obtain $\P(M_n > m_n+x_n) \le (1+o_n(1)) \E[\Lambda_{n,\ell}]$.

Regarding the reverse inequality, note that, by Cauchy--Schwarz inequality, we have
\begin{equation*}
    \P(M_n > m_n+x_n) \ge \P(\Lambda_{n,\ell} \ge 1) \ge \E[\Lambda_{n,\ell}]^2 / \E[\Lambda_{n,\ell}^2].
\end{equation*}
One idea is then to show that the first and second moments of~$\Lambda_{n,\ell}$ are asymptotically equivalent.
An issue is that the second moment of~$\Lambda_{n,\ell}$ may be too large, or even infinite, inflated by rare events where particles make too large a progeny.
To overcome this, we restrict~$\Lambda_{n,\ell}$ to particles that do not have ``too many children too high''.
Our restriction provides a good approximation for which the above second moment method applies.

It then remains to compute an asymptotic equivalent for $\E[\Lambda_{n,\ell}]$.
We achieve this by using results from Doney~\cite{Doney2012} and Pain~\cite{Pain2018} which provide, for a centered random walk $(S_n)_{n \ge 0}$ with finite variance~$\sigma^2$, asymptotic estimates of the form
\begin{equation}\label{eq:random_walk_estimates}
    \P(\forall k \le n, S_k \le x_n, S_n \in (x_n-y-h,x_n-y]) \approx C \frac{x_n y h}{n^{3/2}}\e^{-x_n^2/(2n\sigma^2)},
\end{equation}
for some range of $y>0$ and $h>0$.

Contrary to branching Brownian motion~\cite{Chataignier2025}, if $x_n/\sqrt{n} \to \infty$ and $x_n = o(n)$, then, under our assumptions, the asymptotic behavior of $\P(M_n > m_n+x_n)$ cannot be described with a universal formulation such as in Theorem~\ref{th:moderate_deviation}.
This is due to the fact that, in this case, the left-hand side of~\eqref{eq:random_walk_estimates} depends on the tail distribution of~$S_1$.
The trajectories that give weight to this probability exhibit a full range of behaviors: it can be realized either through a collective contribution of all increments (the \emph{Gaussian regime}), or through the contribution of a single large ``jump'' (the \emph{one big jump regime}), or even through intermediate scenarios involving moderate jumps (see \eg~\cite{BergerBirknerYuan2024}).
In a word, there is no universality anymore.
This suggests that a higher moment condition than~\eqref{eq:assumption_gaussianity} is required.
We believe that it is possible to adapt the strategy of Bramson, Ding, and Zeitouni~\cite{BramsonDingZeitouni2016} for the regime where $x_n/\sqrt{n} \to \infty$ and $x_n = o(n)$ by assuming that there exists $\vep\in(0,1)$ such that $\E[\sum_{|u|=1}\e^{\theta V(u)}]<\infty$ for any $\theta \in [1-\vep,1+\vep]$.
In this paper, in order to state the result under the close-to-optimal assumption~\eqref{eq:assumption_gaussianity}, we only consider the regime where $x_n \to \infty$ and $x_n = O(\sqrt{n})$.

The rest of the paper is organized as follows.
In Section~\ref{sct:preliminaries}, we first introduce the construction of a branching random walk and the spinal decomposition theorem.
Then we introduce some properties of the renewal functions and trajectory estimates of a random walk.
In Section~\ref{sct:integral}, we aim to prove the finiteness of constants~$C_{\mathrm{NL}}$ and~$C_{\mathrm{L}}$ (though their existence is established later), and study the asymptotic behavior of $\int_{D_\ell}y\e^y\P(M_\ell>y)$ as $\ell\to\infty$, for appropriate domains $D_\ell\subset[0,\infty)$.
In Section~\ref{sct:non_contributing_particles}, we exclude some particle behaviors, which are much more atypical than the event $\{M_n>m_n+x_n\}$.
In Section~\ref{sct:equivalent_expectation}, we focus on proving that $\P(M_n>m_n+x_n)$ is asymptotically equivalent to $\E[\Lambda_{n,\ell}]$.
In Section~\ref{sct:asymptotics_for_expectation}, we give sharp upper and lower bounds for $\E[\Lambda_{n,\ell}]$.
In Section~\ref{sct:proof_of_theorem}, we prove the existence of~$C_{\mathrm{NL}}$ and~$C_{\mathrm{L}}$, and complete the proof of Theorem~\ref{th:moderate_deviation}.
Finally, we present two applications of Theorem~\ref{th:moderate_deviation} in Section~\ref{sct:applications}.
The first one is an identification of the region of integration that contributes most to $\int y\e^y\P(M_\ell > y) \d{y}$.
The second one is the convergence in law of the maximum of a two-speed branching random walk in the mean regime.

\section{Preliminaries}\label{sct:preliminaries}

In this section, we give a formal construction of branching random walks.
We then introduce the spinal decomposition theorem and the many-to-one lemma, which relate a branching random walk to a single random walk.
At the end, we gather sharp estimates for random walks.

\subsection{Construction of a branching random walk}

Define the set of finite sequences of positive integers 
\[\T\coloneq\bigcup_{n=0}^\infty \N^{n},\]
with convention $\N^0\coloneq\{\varnothing\}$, where $\N\coloneq\{1,2,3,\ldots\}$.

Given a point process~$\cL$ on~$\R$, let~$\cL^{(u)}$ be \iid copies of~$\cL$.
If there exists some $\theta>0$ such that 
\[\E\left[\sum_{\ell\in\cL}\e^{\theta\ell}\right]<\infty,\]
which is the case with $\theta = 1$ under~\eqref{eq:assumption_boundary_case}, we know that there are only finitely many atoms of~$\cL$ located in $[x,\infty)$ for fixed $x\in\R$.
Therefore, we can rank the atoms, counted with their multiplicity, of~$\cL$ in a non-increasing fashion.
We write~$\cL^{(u)}$ as $(\ell_i^{(u)}:i \ge 1)$, where $(\ell_i^{(u)}:i \ge 1)$ is a non-increasing sequence generated by the atoms in~$\cL^{(u)}$, with the convention that $\ell_i^{(u)}\coloneq-\infty$ if $i>\cL^{(u)}(\R)$.
For $u\in \N^n$, we write $u=(u^{(1)},\ldots,u^{(n)})$ and denote by $|u|$  the generation of~$u$, \ie $|u|=n$.
For $0 \le k \le |u|$, we use $u_k\coloneq(u^{(1)},\ldots,u^{(k)})$ to express the ancestor of~$u$ at time~$k$, with the convention $u_0\coloneq\varnothing$.
With the help of the notation above, we can define a branching random walk $(V(u):u\in\T)$ as follows.
Let $V(\varnothing)=0$, for $u\in \T$ such that $|u| \ge 1$,
\[V(u)\coloneq\sum_{i=0}^{|u|-1}\ell^{(u_i)}_{u^{(i+1)}}.\]
We treat~$u$ such that $V(u)=-\infty$ as dead particles, which do not contribute to the branching random walk.
In the rest of the article, we denote by $\{|u|=n\}$ the set of particles alive at time~$n$.

\subsection{Spinal decomposition: a change of measure}\label{sct:spinal_decomposition}

Here we introduce the spinal decomposition.
This is an important notion for the study of branching processes that we use a lot throughout this paper.
The idea is to describe the law of the branching random walk with an infinite ray of descendants, called spine, and a change of measure.
The main theorem of this section is the spinal decomposition theorem.
Before stating it, we present a similar but more elementary tool, which is also very useful in this paper: the many-to-one formula.

For $a \in \R$, let~$\P_a$ denote the law under which $(V(u) : u \in \T)$ is a branching random walk started from~$a$, and let~$\E_a$ denote the corresponding expectation.
We abbreviate~$\P_0$ and~$\E_0$ by~$\P$ and~$\E$.
Let $(S_n:n \ge 0)$ be a random walk such that, for any non-negative measurable function $h:\R\to[0,\infty)$,
\[\P_a(S_0=a)=1\quad\text{and}\quad\E_a[h(S_1-S_0)]=\E\left[\sum_{|u|=1}h(V(u))\e^{V(u)}\right].\]
We refer to $(S_n : n \ge 0)$ as the random walk associated with the branching random walk.
Note that, under Assumptions~\eqref{eq:assumption_boundary_case} and~\eqref{eq:assumption_gaussianity}, we have $\E[S_1] = 0$ and $\E[S_1^2]=\sigma^2$.

\begin{lemma}[Many-to-one formula, Theorem~1.1 of~\cite{Shi2015}]\label{lem:many_to_one}
For $a\in\R$, $n \ge 1$ and any non-negative measurable function $h:\R^n\to[0,\infty]$, we have 
\[\E_a\left[\sum_{|u|=n}h(V(u_1),\ldots,V(u_n))\e^{V(u)-a}\right]=\E_a\left[h(S_1,\ldots,S_n)\right].\]
\end{lemma}

The many-to-one formula can be showed by induction.
It is also a consequence of the spinal decomposition theorem.
In order to state the latter, let $\hF_n\coloneq\sigma(V(u):|u| \le n)$ and $\hF_\infty\coloneq\sigma(V(u):u\in\T)$, and define
\begin{equation}\label{eq:additive_martingale}
    W_n\coloneq\sum_{|u|=n}\e^{V(u)},\quad n \ge 0.
\end{equation}
It is a martingale with respect to the filtration $(\hF_n:n \ge 0)$ and is usually called the (critical) additive martingale.
By Kolmogorov extension theorem, there exists a probability measure~$\bar{\P}$ on $\hF_\infty$ such that, for any $n \ge 1$ and $A\in\hF_n$, we have
\[\bar{\P}(A)=\E[W_n\mathds{1}_A],\]
with~$\bar{\E}$ the associated expectation.
The branching random walk $(V(u):u\in\T)$ under~$\bar{\P}$ can be seen as a size-biased branching random walk.

Lyons~\cite{Lyons1997} constructed the connection between the size-biased branching random walk and a branching random walk with a spine, which shows the genealogy structure and behavior of particles from a new point of view.
This method is an extension of the spinal decomposition introduced by Lyons, Pemantle, and Peres~\cite{LyonsPemantlePeres1995} for Galton--Watson process and Chauvin and Rouault~\cite{ChauvinRouault1988} for branching Brownian motion.

Let~$\hat{\cL}$ be a point process such that, for any non-negative measurable function~$f$, \[\E[f(\hat{\cL})]=\E\left[f(
\cL)\sum_{\ell\in\cL}\e^{\ell}\right].\]
A branching random walk $(V(u):u\in\T)$ with a spine $(w_n:n \ge 0)$ can be described as follows.
At time~$0$, particle $w_0\coloneq\varnothing$ stays at~$0$.
At each time~$n$, all particles die and give birth to their children independently.
The children of spine particle~$w_{n-1}$ are scattered according to the law of~$\hat{\cL}$, with respect to $V(w_{n-1})$.
The children of normal particles~$u$ are distributed as~$\cL$, with respect to the position of their parents.
The spinal particle~$w_n$ is chosen at random among the children of~$w_{n-1}$, by picking such a child~$v$ with probability proportional to~$\e^{V(v)}$.
We denote by~$\hat{\P}$ the law of the branching random walk with a spine, with~$\hat{\E}$ the associated expectation.
The following theorem tells us that the law of the size-biased branching random walk corresponds to the law of the branching random walk with a spine.

\begin{theorem}[Lyons' spinal decomposition theorem~\cite{Lyons1997}]\label{th:spinal_decomposition}
For $n \ge 0$, the law of $(V(u):|u| \le n)$ is identical under~$\bar{\P}$ and~$\hat{\P}$.
Moreover, for $|u|=n$, we have 
\[\hat{\P}(u=w_n|\hF_n)=\frac{\e^{V(u)}}{W_n},\]
and $(V(w_1),\ldots,V(w_n))$ under~$\hat{\P}$ is distributed as $(S_1,\ldots,S_n)$ under~$\P$.
\end{theorem}

\subsection{Random walk estimates}

We consider a centered random walk $(S_n)_{n \ge 0}$ on~$\R$ with variance $\sigma^2 \in (0, \infty)$.
In order to handle precise estimates, let us introduce some elements from fluctuation theory.
We define the strict ascending ladder epoch processes $(\tau_k^+)_{k \ge 0}$ and ladder height processes $(H_k^+)_{k \ge 0}$ as follows.
We set $\tau_0^+ = 0$, $H_0^+ = 0$, and for $k \ge 1$,
\begin{equation*}
    \tau_k^+ = \inf\{n > \tau_{k-1}^+ : S_n > S_{\tau_{k-1}^+}\} \quad \text{and} \quad H_k^+ = S_{\tau_k^+}.
\end{equation*}
We denote by~$R^+$ the renewal function associated with $(H_k^+)_{k \ge 0}$, defined for $t \ge 0$ by
\begin{equation*}
    R^+(t) = \sum_{k \ge 0} \P(H_k^+ \le t).
\end{equation*}
By Feller~\cite[Theorem~XVIII.5.1~(5.2) and renewal theorem p.~360]{Feller1971}, since $(S_n)$ is centered with finite variance, we have $\E \left[ H_1 \right] < \infty$, which implies that there exists $c^+ > 0$ such that
\begin{equation}\label{eq:feller}
    \lim_{t \to \infty} \frac{R^+(t)}{t} = c^+.
\end{equation}
By Kozlov~\cite[Theorem~A]{Kozlov1976}, there exists $\theta^+ > 0$ such that, for any $t \ge 0$,
\begin{equation*}
    \P(\forall k \le n, S_k \le t) \sim \frac{\theta^+ R^+(t)}{\sqrt{n}},
\end{equation*}
as $n \to \infty$.
A\"id\'ekon and Shi~\cite[Lemma~2.1]{AidekonShi2014} have showed that constants~$c^+$ and~$\theta^+$ are related through
\begin{equation}\label{eq:aidekon_shi}
    \theta^+ c^+ = \sqrt{\frac{2}{\pi \sigma^2}}.
\end{equation}
We also consider the strict descending ladder variables~$\tau_k^-$, $H_k^-$, $R^-$, $c^-$, $\theta^-$, the analogous quantities associated to the random walk $(-S_n)_{n \ge 0}$.

The following lemma is from Doney~\cite[Propositions~11, 18, and~24]{Doney2012}\footnote{Doney states the result in terms of the renewal function associated to the \emph{weak} descending ladder height process.
The link with the \emph{strict} descending ladder height process is explained by~Caravenna and Chaumont~\cite[Remark~4.6]{CaravennaChaumont2013}.
It leads to our formulation, which was already given in~\cite[Lemma~B.1]{Pain2018}.}.

\begin{lemma}[Doney~\cite{Doney2012}]\label{lem:doney}
    Let $D > 0$ and $(\gamma_n)_{n \ge 0}$ be any real sequence such that $0 < \gamma_n = o(\sqrt{n})$.
    Define, for $x \in \R$, $g(x) = x \e^{-x^2/2} \mathds{1}_{\{x>0\}}$.
    \begin{enumerate}
        \item If $(S_n)$ is non-lattice, then for any $h > 0$,
        \begin{equation*}
            \P(\forall k \le n, S_k \le x, S_n \in (x-y-h, x-y]) \sim \frac{\theta^-}{\sigma n} g \left( \frac{x}{\sigma \sqrt{n}} \right) \int_y^{y+h} R^-(t) \d{t},
        \end{equation*}
        uniformly in $x \in [D^{-1}\sqrt{n}, D\sqrt{n}]$ and in $y \in [0, \gamma_n]$.
        \item If $(S_n)$ is $(h,a)$-lattice, then
        \begin{equation*}
            \P(\forall k \le n, S_k \le x, S_n = x-y) \sim \frac{\theta^-}{\sigma n} g \left( \frac{x}{\sigma \sqrt{n}} \right) h R^-(y),
        \end{equation*}
        uniformly in $x \in [D^{-1}\sqrt{n}, D\sqrt{n}]$ and in $y \in [0, \gamma_n]$ such that $x-y \in an+h\Z$.
    \end{enumerate}
\end{lemma}

The following lemma is a direct consequence of Pain~\cite[Proposition~2.8]{Pain2018}.

\begin{lemma}[Pain~\cite{Pain2018}]\label{lem:pain}
    Let $\lambda \in (0,1)$ and $(\gamma_n)_{n \ge 0}$ be any real sequence such that $0 < \gamma_n = o(\sqrt{n})$.
    \begin{enumerate}
        \item If $(S_n)$ is non-lattice, then, for any $h > 0$,
        \begin{equation*}
            \P(\forall k \le \lfloor \lambda n \rfloor, S_k \le x, \forall \lfloor \lambda n \rfloor \le j \le n, S_j \le x', S_n \in (x'-y-h, x'-y]) \sim \sqrt{\frac{\pi}{2}} \frac{\theta^+ \theta^-}{\sigma} \frac{R^+(x)}{n^{3/2}} \int_y^{y+h} R^-(t) \d{t},
        \end{equation*}
        uniformly in $x, y \in [0,\gamma_n]$ and $x' \in [-\gamma_n,\gamma_n]$.
        \item If $(S_n)$ is $(h,a)$-lattice, then
        \begin{equation*}
            \P(\forall k \le \lfloor \lambda n \rfloor, S_k \le x, \forall \lfloor \lambda n \rfloor \le j \le n, S_j \le x', S_n = x'-y) \sim \sqrt{\frac{\pi}{2}} \frac{\theta^+ \theta^-}{\sigma} \frac{R^+(x)}{n^{3/2}} h R^-(y),
        \end{equation*}
        uniformly in $x,y \in [0, \gamma_n]$ and $x' \in [-\gamma_n,\gamma_n]$ such that $x'-y \in an+h\Z$.
    \end{enumerate}
\end{lemma}

In this paper, we will consider the asymptotics given in Lemmas~\ref{lem:doney} and~\ref{lem:pain} when $x,y,x' \to \infty$.
In order to let the variables go to infinity one after another, it will be convenient to state the asymptotics in the following way.

\begin{corollary}\label{cor:simple_barrier_estimate}
    Let $D > 0$ and $(\gamma_n)_{n \ge 0}$ be a real sequence such that $0 < \gamma_n = o(\sqrt{n})$.
    Let~$h$ be any positive real number if~$S_1$ is non-lattice or the lattice span if~$S_1$ is lattice.
    Then,
    \begin{multline}\label{eq:simple_barrier_estimate}
        \P(\forall k \le n, S_k \le x, S_n \in (x-y-h, x-y]) \\
        = (1 + o_{n \to \infty}(1)) (1 + o_{x \to \infty}(1)) (1 + o_{y \to \infty}(1)) \frac{1}{\sigma^3} \sqrt{\frac{2}{\pi}} \frac{xyh}{n^{3/2}} \e^{-x^2/(2n\sigma^2)},
    \end{multline}
    uniformly in $x \in [0,D\sqrt{n}]$ and $y \in [0,\gamma_n]$.
\end{corollary}

\begin{proof}
    Let us first treat the non-lattice case.
    Define
    \begin{equation}\label{eq:simple_barrier_estimate_non_lattice}
        \vep_1(x) = \sqrt{\frac{\pi\sigma^2}{2}} \frac{\theta^+ R^+(x)}{x} - 1 \quad \text{and} \quad \vep_2(y) = \sqrt{\frac{\pi\sigma^2}{2}} \frac{\theta^-}{yh} \int_y^{y+h} R^-(t) \d{t} - 1.
    \end{equation}
    By~\eqref{eq:feller} and~\eqref{eq:aidekon_shi}, we have $\vep_1(x) = o_{x \to \infty}(1)$ and $\vep_2(y) = o_{y \to \infty}(1)$.
    Now, let $(x_n)$ and $(y_n)$ be non-negative sequences such that $y_n = o(\sqrt{n})$ and
    \begin{enumerate}
        \item either $x_n \asymp \sqrt{n}$,
        \item or $x_n = o(\sqrt{n})$.
    \end{enumerate}
    Applying Lemma~\ref{lem:doney} if $x_n \asymp\sqrt{n}$ and Lemma~\ref{lem:pain} if $x_n = o(\sqrt{n})$, we obtain that, as $n \to \infty$,
    \begin{equation*}
        \P(\forall k \le n, S_k \le x_n, S_n \in (x_n-y_n-h, x_n-y_n]) \sim (1 + \vep_1(x_n)) (1 + \vep_2(y_n)) \frac{1}{\sigma^3} \sqrt{\frac{2}{\pi}} \frac{x_n y_n h}{n^{3/2}} \e^{-x_n^2/(2n\sigma^2)}.
    \end{equation*}
    By Lemma~\ref{lem:combine_regimes}, we deduce that it still holds for any sequences $(x_n)$ and $(y_n)$ such that $0 \le x_n = O(\sqrt{n})$ and $0 \le y_n = o(\sqrt{n})$.
    The asymptotic~\eqref{eq:simple_barrier_estimate} then follows from Lemma~\ref{lem:from_sequence_to_uniformity}.

    The $(h,a)$-lattice case is handled in the same way, except that we have to slightly modify~$\vep_2(y)$ to take the lattice into account.
    Note that, in this case,
    \begin{equation*}
        \P(\forall k \le n, S_k \le x, S_n \in (x-y-h, x-y]) = \P(\forall k \le n, S_k \le x, S_n = [x-y]_n),
    \end{equation*}
    where $[z]_n$ denotes the greatest number in $an+h\Z$ such that $[z]_n \le z$.
    Define
    \begin{equation}\label{eq:simple_barrier_estimate_lattice}
        \vep_2(n,x,y) = \sqrt{\frac{\pi\sigma^2}{2}} \frac{\theta^- R^-(x-[x-y]_n)}{x-[x-y]_n} - 1.
    \end{equation}
    Since $y \le x-[x-y]_n < y+h$, we still have $\vep_2(n,x,y) = o_{y \to \infty}(1)$, by~\eqref{eq:feller} and~\eqref{eq:aidekon_shi}.
    The rest of the proof is identical to that of the non-lattice case.
\end{proof}

\begin{corollary}\label{cor:double_barrier_estimate}
    Let $\lambda \in (0,1)$, $D > 0$, and $(\gamma_n)_{n \ge 0}$ be a real sequence such that $0 < \gamma_n = o(\sqrt{n})$.
    Let~$h$ be any positive real number if~$S_1$ is non-lattice or the lattice span if~$S_1$ is lattice.
    Then,
    \begin{multline*}
        \P(\forall k \le \lfloor \lambda n \rfloor, S_k \le x, \forall \lfloor \lambda n \rfloor \le j \le n, S_j \le x-z, S_n \in (x-z-y-h, x-z-y]) \\
        \ge (1 + o_{n \to \infty}(1)) (1 + o_{x \to \infty}(1)) (1 + o_{y \to \infty}(1)) \frac{1}{\sigma^3} \sqrt{\frac{2}{\pi}} \frac{xyh}{n^{3/2}} \e^{-x^2/(2n\sigma^2)},
    \end{multline*}
    uniformly in $x \in [0,D\sqrt{n}]$ and $y,z \in [0,\gamma_n]$.
\end{corollary}

\begin{proof}
    Define~$\vep_1(x)$, $\vep_2(y)$, and~$\vep_2(n,x,y)$ as in~\eqref{eq:simple_barrier_estimate_non_lattice} and~\eqref{eq:simple_barrier_estimate_lattice}.
    To unify the non-lattice and lattice cases, let $\tilde{\vep}_2(n,x,y) = \vep_2(y)$ in the former and $\tilde{\vep}_2(n,x,y) = \vep_2(n,x,y)$ in the latter.
    Similarly to the proof of Corollary~\ref{cor:simple_barrier_estimate}, let $(x_n)$, $(y_n)$, and $(z_n)$ be non-negative sequences such that $y_n,z_n = o(\sqrt{n})$.
    If $x_n = o(\sqrt{n})$, then Lemma~\ref{lem:pain} implies that
    \begin{multline}\label{eq:double_barrier_estimate_1}
        \P(\forall k \le \lfloor \lambda n \rfloor, S_k \le x_n, \forall \lfloor \lambda n \rfloor \le j \le n, S_j \le x_n-z_n, S_n \in (x_n-z_n-y_n-h, x_n-z_n-y_n]) \\
        = (1 + o_{n \to \infty}(1)) (1 + \vep_1(x_n)) (1 + \tilde{\vep}_2(n,x_n,y_n)) \frac{1}{\sigma^3} \sqrt{\frac{2}{\pi}} \frac{x_n y_n h}{n^{3/2}} \e^{-x_n^2/(2n\sigma^2)}.
    \end{multline}
    If $x_n \asymp \sqrt{n}$, then the left-hand side of~\eqref{eq:double_barrier_estimate_1} is higher than
    \begin{multline}\label{eq:double_barrier_estimate_2}
        \P(\forall k \le n, S_k \le x_n-z_n, S_n \in (x_n-z_n-y_n-h, x_n-z_n-y_n]) \\
        = (1 + o_{n \to \infty}(1)) (1 + \vep_1(x_n-z_n)) (1 + \tilde{\vep}_2(n,x_n-z_n,y_n)) \frac{1}{\sigma^3} \sqrt{\frac{2}{\pi}} \frac{(x_n-z_n) y_n h}{n^{3/2}} \e^{-(x_n-z_n)^2/(2n\sigma^2)},
    \end{multline}
    where the equality follows from Lemma~\ref{lem:doney}.
    Note that, in this regime, we have $x_n-z_n \sim x_n \to \infty$ and then we can replace the occurrences of $x_n-z_n$ by~$x_n$ in the right-hand side of~\eqref{eq:double_barrier_estimate_2}.
    The result then follows from Lemmas~\ref{lem:combine_regimes} and~\ref{lem:from_sequence_to_uniformity}.
\end{proof}

We will also need to estimate the probability for a random walk to make an excursion below a curved barrier.
To this end, fix $A > 0$ and $\alpha < 1/2$.
Consider some functions $k \mapsto f_n(k)$, for $n \ge 1$, such that

\begin{equation}\label{eq:condition_f}
    f_n(0) = O_n(1) \quad \text{and} \quad \sup_{1 \le j \le n} \max\left( \frac{|f_n(j)|}{j^{\alpha}}, \frac{|f_n(n)-f_n(j)|}{(n-j)^\alpha} \right) \le A.
\end{equation}
The following estimate is from Mallein~\cite[Lemma~3.1]{Mallein2016}, except that we do not assume $f_n(0) = 0$.

\begin{lemma}[\cite{Mallein2016}]\label{lem:mallein_random_walk_estimate}
    Assume~\eqref{eq:condition_f}.
    Then there exists $C > 0$ such that, for any $x, y, z > 0$,
    \begin{equation*}
        \P(\forall k \le n, S_k \le x+f_n(k), S_n \in x+f_n(n)-[y-z,y]) \le C \frac{(1 + x \wedge \sqrt{n})(1 + y \wedge \sqrt{n})(1 + z \wedge \sqrt{n})}{n^{3/2}}.
    \end{equation*}
\end{lemma}

\begin{proof}
    Applying~\cite[Lemma~3.1]{Mallein2016} with the barrier $k \mapsto f_n(k) - f_n(0)$, we obtain
    \begin{equation*}
        \P(\forall k \le n, S_k \le x+f_n(k), S_n \in x+f_n(n)-[y-z,y]) \le C \frac{(1 + (x + f_n(0))_+ \wedge \sqrt{n})(1 + y \wedge \sqrt{n})(1 + z \wedge \sqrt{n})}{n^{3/2}}.
    \end{equation*}
    We then bound $(1 + (x + f_n(0))_+ \wedge \sqrt{n}) \le (1 + x \wedge \sqrt{n})(1 + |f_n(0)|)$.
    Since $f_n(0) = O_n(1)$, we can absorb the last factor in the constant~$C$.
\end{proof}

We will also need the following variation of Lemma~\ref{lem:mallein_random_walk_estimate} in the proof of the forthcoming Lemma~\ref{lem:peeling_lemma}.

\begin{lemma}\label{lem:ballot_additional_variables}
    Assume~\eqref{eq:condition_f}.
    Then, there exists $C > 0$ such that the following holds.
    Consider additional random variables~$\xi_k$, for $k \ge 1$, such that $(S_k-S_{k-1}, \xi_k)$ are \iid random variables on~$\R^2$.
    For any $n \ge 1$ and $x, y \ge 0$,
    \begin{equation}\label{eq:ballot_additional_variables}
        \P(\forall k \le n, S_k \le x+f_n(k), \exists j < n, S_j > x+f_n(j)-\xi_{j+1}, S_n \in x+f_n(n)-[y,y+1]) \le C \eta \frac{(1+x)(1+y)}{n^{3/2}},
    \end{equation}
    where
    \begin{equation*}
        \eta = \P(\xi_1 - S_1 \ge 0)^{1/2} + \E \left[ (\xi_1 - S_1)_+^2 \right]^{1/2} + \P(\xi_1+1 \ge 0)^{1/2} + \E \left[ (\xi_1+1)_+^2 \right]^{1/2}.
    \end{equation*}
\end{lemma}
We postpone the proof of Lemma~\ref{lem:ballot_additional_variables} to the Appendix~\ref{sct:ballot_additional_variables}.
It is similar to~\cite[Lemma~B.2]{Aidekon2013}.

\section{Asymptotic properties of the integral defining~\texorpdfstring{$C^*$}{C*}}\label{sct:integral}

The aim of this section is to study the asymptotic properties, as $\ell \to \infty$, of the integral and the sum defining constants~$C_{\mathrm{NL}}$ and~$C_{\mathrm{L}}$ in Theorem~\ref{th:moderate_deviation}, or equivalently, of the integral defining~$C^*$ in Remark~\ref{rem:unification}, that is
\begin{equation}\label{eq:integral}
    \int_0^\infty y \e^y \P(M_\ell > y) \d{y}.
\end{equation}
To show that~$C_{\mathrm{NL}}$ and~$C_{\mathrm{L}}$ are positive constants, we must prove that~\eqref{eq:integral} converges in $(0,\infty)$ as $\ell \to \infty$.
As a first step, we obtain that the $\limsup$ and $\liminf$ are indeed in $(0,\infty)$.
The proof of the convergence is postponed to Section~\ref{sct:proof_of_theorem} since it will follow from an argument developed in Section~\ref{sct:asymptotics_for_expectation}.

It will also be useful to identify the main support of the integral in~\eqref{eq:integral}.
One can expect it to correspond to values of~$y$ that are of order~$\sqrt{\ell}$, \eg in view of equation~\eqref{eq:moderate_deviation_bbm} for the branching Brownian motion.
However, we do not need to be that precise for the proof of our main result Theorem~\ref{th:moderate_deviation}.
It will be sufficient to show that the integration domain can be restricted to $[\vep\sqrt{\ell},B\ell]$, with $\vep \to 0$ and $B \to \infty$ after $\ell \to \infty$.
This is achieved in the following two propositions.
The fact that we can restrict the domain to $[\vep\sqrt{\ell},B\sqrt{\ell}]$ is stated in Section~\ref{sct:integral_2} as a consequence of Theorem~\ref{th:moderate_deviation}.

\begin{proposition}\label{prop:restriction_lower}
Assume~\eqref{eq:assumption_supercriticality}, \eqref{eq:assumption_boundary_case}, and~\eqref{eq:assumption_gaussianity}.
Then,
\[\limsup_{\vep \to 0^+}\limsup_{\ell\to\infty}\int_0^{\vep\sqrt{\ell}}y \e^y\P(M_\ell>y)\d{y}=0.\]
\end{proposition}

\begin{proof}
According to~\eqref{eq:hu_upper_bound}, we know that there exists $C>0$ such that, for any $\ell \ge 1$ and $z \ge 0$, we have
\[\P(M_\ell>m_\ell+z) \le C(1+z)\e^{-z}.\]
This implies that
\[\int_0^{\vep\sqrt{\ell}}y \e^y\P(M_\ell>y)\d{y} \le C\int_0^{\vep\sqrt{\ell}}y \e^y (y-m_\ell)\e^{-y+m_\ell}\d{y} \le \frac{C}{\ell^{3/2}}\int_0^{\vep\sqrt{\ell}}y^2\d{y}+\frac{C\log \ell}{\ell^{3/2}}\int_0^{\vep\sqrt{\ell}}y\d{y}.\]
Making $\ell \to \infty$ and then $\vep \to 0^+$, this bound goes to $0$.
\end{proof}

\begin{proposition}\label{prop:restriction_upper}
Assume~\eqref{eq:assumption_supercriticality}, \eqref{eq:assumption_boundary_case}, and~\eqref{eq:assumption_gaussianity}.
Then,
\[\lim_{B\to\infty}\sup_{\ell \ge 1}\int_{B\ell}^\infty y\e^y\P(M_\ell>y)\d{y}=0.\]
\end{proposition}

\begin{proof}
By Fubini's theorem, we have
\[\int_{B\ell}^\infty y\e^y\P(M_\ell>y)\d{y} \le \E\left[M_\ell \e^{M_\ell}\mathds{1}_{\{M_\ell>B\ell\}}\right].\]
And we know that
\begin{equation*}
\E\left[M_\ell \e^{M_\ell}\mathds{1}_{\{M_\ell>B\ell\}}\right] \le \E\left[\sum_{|u|=\ell}V(u)\e^{V(u)}\mathds{1}_{\{V(u)>B\ell\}}\right] = \E\left[S_\ell \mathds{1}_{\{S_\ell>B\ell\}}\right],
\end{equation*}
where the equality follows from the many-to-one formula.
By Markov's inequality, we have 
\[\E\left[S_\ell \mathds{1}_{\{S_\ell>B\ell\}}\right] \le \frac{\E\left[S_\ell^2\right]}{B\ell}=\frac{\sigma^2}{B}.\]
Therefore, we conclude that, for any $B>0$,
\[\sup_{\ell \ge 1}\int_{B\ell}^\infty y\e^y\P(M_\ell>y)\d{y} \le \frac{\sigma^2}{B}.\]
Letting $B\to\infty$, this completes the proof.
\end{proof}

Up to the existence of constants~$C_{\mathrm{NL}}$ and~$C_{\mathrm{L}}$ defined in Theorem~\ref{th:moderate_deviation} (proved in Section~\ref{sct:proof_of_theorem} below), the following proposition implies that they are finite.

\begin{proposition}\label{prop:finite_integral} Assume~\eqref{eq:assumption_supercriticality}, \eqref{eq:assumption_boundary_case}, and~\eqref{eq:assumption_gaussianity}.
Then,
\begin{equation}\label{eq:finite_integral}
    \limsup_{\ell\to\infty}\int_0^\infty y \e^y\P(M_\ell > y)\d{y}<\infty.
\end{equation}
\end{proposition}

To prove Proposition~\ref{prop:finite_integral}, we will need the following lemma, which will also be useful throughout the rest of the paper.

\begin{lemma}\label{lem:negligible_integral}
Assume~\eqref{eq:assumption_supercriticality}, \eqref{eq:assumption_boundary_case}, and~\eqref{eq:assumption_gaussianity}.
Then, as $\ell \to \infty$,
\[\int_0^\infty \e^y\P(M_\ell>y)\d{y}=O(\ell^{-1/4}).\]
\end{lemma}

\begin{proof}
We split $[0,\infty)$ into two parts $[0,a_\ell)$ and $[a_\ell,\infty)$, for $a_\ell$ that will be fixed later on.
On the one hand, according to~\eqref{eq:hu_upper_bound}, we have
\[\int_0^{a_\ell} \e^y\P(M_\ell>y)\d{y} \le \frac{C}{\ell^{3/2}}\int_0^{a_\ell}(y-m_\ell)\d{y} \le \frac{Ca_\ell^2}{\ell^{3/2}}+\frac{Ca_\ell\log \ell}{\ell^{3/2}}.\]
On the other hand, we have
\begin{equation*}
\int_{a_\ell}^\infty \e^y\P(M_\ell>y)\d{y} \le \E \left[ \e^{M_\ell}\mathds{1}_{\{M_\ell \ge a_\ell\}} \right] \le \E \left[ \sum_{|u|=\ell}\e^{V(u)}\mathds{1}_{\{V(u) \ge a_\ell\}} \right] = \P(S_\ell \ge a_\ell) \le \frac{1}{a_\ell^2} \E \left[ S_\ell^2 \right] = \frac{\sigma^2 \ell}{a_\ell^2},
\end{equation*}
where the first equality follows from the many-to-one formula.
Taking $a_\ell \coloneq \ell^{5/8}$, this completes the proof.
\end{proof}

Before turning to the proof of Proposition~\ref{prop:finite_integral}, let us introduce some notation that will be useful there and in the rest of the paper.
For $u,v\in\T$, we write $v \succeq u$ if~$v$ is~$u$ or one of its descendants.
Besides, given $n \ge 0$, we define
\begin{equation}\label{eq:relative_max}
    M_n^{(u)}=\max_{|v|=|u|+n,v \succeq u} (V(v)-V(u)),
\end{equation}
the maximal position (with respect to $V(u)$) of a branching random walk generated from~$u$.

\begin{proof}[Proof of Proposition~\ref{prop:finite_integral}]
    Define, for $n \ge 1$ and $0 \le k \le n$, $f_n(k) = -\frac{3}{2} \log \frac{n}{n-k+1}$ and, for $x \ge 0$ and $0 \le \ell \le n$,
    \begin{equation*}
        \Gamma_{n,\ell}(x) = \sum_{|u|=n-\ell} \mathds{1}_{\{\forall j \le n-\ell, V(u_j) \le f_n(j)\}} \mathds{1}_{\{M_\ell^{(u)} > m_n+x\}}.
    \end{equation*}
    We will control the left-hand side of~\eqref{eq:finite_integral} thanks to $\E[\Gamma_{n,\ell}(x)]$.
    First, let us state an upper bound for the latter quantity: by Markov's inequality,
    \begin{equation}\label{eq:finite_integral_mallein}
        \E[\Gamma_{n,\ell}(x)] \le \sum_{k=n-\ell+1}^n \E \left[ \sum_{|u|=k} \mathds{1}_{\{\forall j < k, V(u_j) \le f_n(j)\}} \mathds{1}_{\{V(u) > f_n(k)+x\}} \right] \le C(1+x)\e^{-x},
    \end{equation}
    where the last inequality is shown by Mallein in the proof of~\cite[Lemma~4.2]{Mallein2016}, using Lemmas~\ref{lem:many_to_one} and~\ref{lem:mallein_random_walk_estimate}.
    
    Now, we apply successively the Markov property and the many-to-one formula to write
    \begin{align*}
        \E[\Gamma_{n,\ell}(x)]
        & = \E \left[ \sum_{|u|=n-\ell} \mathds{1}_{\{\forall j \le n-\ell, V(u_j) \le f_n(j)+x\}} \P_{V(u)}(M_\ell>m_n+x) \right] \\
        & = \E \left[ \e^{-S_{n-\ell}} \mathds{1}_{\{\forall j \le n-\ell, S_j \le f_n(j)+x\}} \P_{S_{n-\ell}}(M_\ell>m_n+x) \right] \\
        & \ge  \E \left[ \e^{-S_{n-\ell}} \mathds{1}_{\{\forall j \le n-\ell, S_j \le m_n+x\}} \P_{S_{n-\ell}}(M_\ell>m_n+x) \right].
    \end{align*}
    Fix~$h$ any positive number if~$S_1$ is non-lattice, the lattice span if~$S_1$ is $(h,a)$-lattice.
    Decompose the event $\{S_{n-\ell} \le m_n+x\}$ to write
    \begin{align*}
        \E[\Gamma_{n,\ell}(x)]
        & \ge  \sum_{i \ge 0} \E \left[ \e^{-S_{n-\ell}} \mathds{1}_{\left\{ \substack{\forall j < n-\ell, S_j \le m_n+x, \\ S_{n-\ell}-m_n-x \in (-(i+1)h,-ih]} \right\}} \P_{S_{n-\ell}}(M_\ell>m_n+x) \right] \\
        & \ge \sum_{i \ge 0} \E \left[ \e^{-m_n-x+ih} \mathds{1}_{\left\{\substack{\forall j < n-\ell, S_j \le m_n+x, \\ S_{n-\ell}-m_n-x \in (-(i+1)h,-ih]} \right\}} \P(M_\ell>(i+1)h) \right].
    \end{align*}
    In view of applying Corollary~\ref{cor:simple_barrier_estimate}, we choose $x = x_n$ such that $x_n \to \infty$ and $x_n = o(\sqrt{n})$ as $n \to \infty$, and we restrict the sum to $i \ge A_\ell/h$, where $A_\ell$ is such that $A_\ell \to \infty$ and $A_\ell = o(\sqrt{\ell})$ as $\ell \to \infty$.
    This, together with Fatou's lemma, yields
    \begin{equation}\label{eq:finite_integral_fatou}
        \liminf_{n \to \infty} \frac{\E[\Gamma_{n,\ell}(x_n)]}{x_n \e^{-x_n}} \ge \sum_{i \ge A_\ell/h} \e^{ih} \P(M_\ell > (i+1)h) \liminf_{n \to \infty} \frac{n^{3/2}}{x_n} \P \left( \substack{\forall j < n-\ell, S_j \le m_n+x_n, \\ S_{n-\ell}-m_n-x_n \in (-(i+1)h,-ih]} \right).
    \end{equation}
    Applying Corollary~\ref{cor:simple_barrier_estimate}, it becomes
    \begin{align}
        \liminf_{n \to \infty} \frac{\E[\Gamma_{n,\ell}(x_n)]}{x_n \e^{-x_n}}
        & \ge \frac{(1 + o_\ell(1))}{\sigma^3} \sqrt{\frac{2}{\pi}} \sum_{i \ge A_\ell/h} \e^{ih} \P(M_\ell > (i+1)h) ih^2 \nonumber \\
        & \ge \frac{(1+o_\ell(1))}{\sigma^3} \sqrt{\frac{2}{\pi}} \int_{A_\ell+h}^\infty \e^{y-2h} \P(M_\ell > y) (y-2h) \d{y}, \label{eq:finite_integral_restriction}
    \end{align}
    by monotonicity.
    Letting $\ell \to \infty$ and applying successively~\eqref{eq:finite_integral_restriction}, Lemma~\ref{lem:negligible_integral}, and Proposition~\ref{prop:restriction_lower}, we obtain
    \begin{align}
        \limsup_{\ell \to \infty} \liminf_{n \to \infty} \frac{\E[\Gamma_{n,\ell}(x_n)]}{x_n \e^{-x_n}}
        & \ge \frac{1}{\sigma^3} \sqrt{\frac{2}{\pi}} \limsup_{\ell \to \infty} \int_{A_\ell+h}^\infty \e^{y-2h} \P(M_\ell > y) (y-2h) \d{y} \nonumber \\
        & \ge \frac{\e^{-2h}}{\sigma^3} \sqrt{\frac{2}{\pi}} \limsup_{\ell \to \infty} \int_{A_\ell+h}^\infty y \e^y \P(M_\ell > y) \d{y} \nonumber \\
        & = \frac{\e^{-2h}}{\sigma^3} \sqrt{\frac{2}{\pi}} \limsup_{\ell \to \infty} \int_0^\infty y \e^y \P(M_\ell > y) \d{y}. \label{eq:finite_integral_limsup}
    \end{align}
    Combined with~\eqref{eq:finite_integral_mallein}, this implies~\eqref{eq:finite_integral}.
\end{proof}

Let us conclude this section with a straightforward consequence of the lower bound~\eqref{eq:mallein_lower_bound}, which implies the positivity of constants~$C_{\mathrm{NL}}$ and~$C_{\mathrm{L}}$, up to their existence (proved in Section~\ref{sct:proof_of_theorem} below).

\begin{proposition}\label{prop:positivity}
Assume~\eqref{eq:assumption_supercriticality}, \eqref{eq:assumption_boundary_case}, \eqref{eq:assumption_gaussianity}, and~\eqref{eq:assumption_peeling_lemma}.
Then,
\[\liminf_{\ell\to\infty}\int_0^\infty y\e^y\P(M_\ell>y)\d{y}>0.\]
\end{proposition}

\begin{proof}
By~\eqref{eq:mallein_lower_bound}, there exists $c>0$ such that, for any $\ell \ge 1$ and $y \in [0,\sqrt{\ell}]$,
\[\P(M_\ell>y) \ge \frac{c(1+y-m_\ell)}{\ell^{3/2}}\e^{-y}.\]
Hence, for any $\ell \ge 1$,
\[\int_0^\infty y\e^y\P(M_\ell>y)\d y \ge \int_0^{\sqrt{\ell}}y\e^y\P(M_\ell>y)\d y \ge c\ell^{-3/2}\int_0^{\sqrt{\ell}}y^2 \d y \ge \frac{c}{3},
\]
which completes the proof.
\end{proof}

\section{Non-contributing particles}\label{sct:non_contributing_particles}

Suppose that $(x_n)_{n \ge 0}$ is such that $x_n \to \infty$ and $x_n = O(\sqrt{n})$.
As in the proof of Proposition~\ref{prop:finite_integral}, define, for $n \ge 1$ and $0 \le k \le n$,
\begin{equation}\label{eq:def_barrier}
    f_n(k) = m_n-m_{n-k+1} = -\frac{3}{2} \log \frac{n}{n-k+1}.
\end{equation}
The main result of this section is Proposition~\ref{prop:upper_bound_G}.
It tells that particles contributing to the upper moderate deviations typically stay below the barrier $x_n + f_{n-\ell}(k)$ until time~$k = n - \ell$, for some large~$\ell$.
To formalize it, we define, for $0 \le \ell \le n$,
\begin{equation}\label{eq:def_G}
    G_{n,\ell} = \{ \exists |u|=n, V(u) > m_n+x_n, \exists k \le n-\ell, V(u_k) > x_n+f_{n-\ell}(k) \}.
\end{equation}
Instead of taking~$\ell$ a large constant, it will be convenient to consider a sequence $(\ell_n)$ that goes to infinity arbitrarily slowly.

\begin{proposition}\label{prop:upper_bound_G}
  Assume~\eqref{eq:assumption_supercriticality}, \eqref{eq:assumption_boundary_case}, and~\eqref{eq:assumption_gaussianity}.
  If $\ell_n \to \infty$ and $\ell_n = o(n)$, then $\P(G_{n,\ell_n}) = o(x_n\e^{-x_n})$.
\end{proposition}

\begin{proof}
    To simplify notations, let us abbreviate $\ell = \ell_n$.
    Recall that $M_n^{(u)}$ is defined in~\eqref{eq:relative_max}.
    By the union bound,
    \begin{equation*}
        \P(G_{n,\ell}) \le \sum_{k=1}^{n-\ell} \E \left[ \sum_{|u|=k} \mathds{1}_{\left\{ \substack{\forall j < k, V(u_j) \le x_n+f_{n-\ell}(j), \\ V(u) > x_n+f_{n-\ell}(k), \\ M_{n-k}^{(u)} + V(u) > m_n + x_n} \right\}} \right] = \sum_{k=1}^{n-\ell} u_n(k),
    \end{equation*}
    where, by the Markov property at time~$k$ and the many-to-one formula,
    \begin{equation*}
        u_n(k) = \E \left[ \e^{-S_k} \mathds{1}_{\left\{ \substack{\forall j < k, S_j \le x_n+f_{n-\ell}(j), \\ S_k > x_n+f_{n-\ell}(k)} \right\}} \P_{S_k}(M_{n-k} > m_n + x_n) \right].
    \end{equation*}
    Fix some sequence of integers $(k_n)$ such that $k_n \to \infty$ and $k_n = o(x_n)$.
    First assume $k \le k_n$.
    Observe that, bounding the probabilities by~$1$, we obtain $u_n(k) \le \e^{-x_{n} + f_{n-\ell}(k)}$.
    In particular, summing~$u_n(k)$ up to~$k_n$ yields
    \begin{equation*}
        \sum_{k=1}^{k_n} u_n(k) \le k_n \e^{-x_n-f_{n-\ell}(k_n)} = o(x_n\e^{-x_n}).
    \end{equation*}
    Now, assume $k \in [k_n,n-\ell]$.
    By~\eqref{eq:hu_upper_bound},
    \begin{equation*}
        \P_{S_k}(M_{n-k} > m_n + x_n) \le C(1+(m_n-m_{n-k}+x_n-S_k)_+) \e^{-(m_n-m_{n-k}+x_n-S_k)}.
    \end{equation*}
    Note that 
    \[
    m_n-m_{n-k}-f_{n-\ell}(k) =-\frac{3}{2}\log \frac{n}{n-\ell} + \frac{3}{2}\log \frac{n-k}{n-\ell-k+1} \le \frac{3}{2}\log \frac{n-k}{n-\ell-k+1} = |f_{n-k}(\ell)|.
    \]
    It follows that
    \begin{equation*}
        u_n(k) \le C (1+|f_{n-k}(\ell)|) \e^{-(m_n-m_{n-k}+x_n)} \P \left( \substack{\forall j < k, S_j \le x_n+f_{n-\ell}(j), \\ S_k > x_n+f_{n-\ell}(k)} \right).
    \end{equation*}
    By the Markov property at time $k-1$, the last factor is equal to $\E \left[ \phi_k(S_1) \right]$, where 
    \begin{equation*}
        \phi_k(y) = \P\left( \substack{\forall j < k, S_j \le x_n+f_{n-\ell}(j), \\ S_{k-1} > x_n+f_{n-\ell}(k) - y} \right) \le \frac{C(1+x_n)(1+y_+)^2}{k^{3/2}},
    \end{equation*}
    by Lemma~\ref{lem:mallein_random_walk_estimate}.
    By Assumption~\eqref{eq:assumption_gaussianity}, $S_1$ admits a finite second moment and then $\E \left[ \phi_k(S_1) \right] \le Cx_n/k^{3/2}$.
    Hence,
    \begin{equation*}
        \sum_{k=k_n}^{n-\ell} u_n(k) \le C x_n \e^{-x_n} \sum_{k=k_n}^{n-\ell} (1+|f_{n-k}(\ell)|)\frac{n^{3/2}}{(n-k)^{3/2}k^{3/2}} \le  C x_n \e^{-x_n} \left( \sum_{k=k_n}^{n/2} \frac{|f_{n-k}(\ell)|}{k^{3/2}} + \sum_{k=\ell}^{n/2} \frac{|f_k(\ell)|}{k^{3/2}} \right).
    \end{equation*}
    Note that, for $k \le n/2$, we have $|f_{n-k}(\ell)| \le C$, and for $k \ge \ell$, we have $|f_k(\ell)| \le (3/2) \log \ell$.
    Hence,
    \begin{equation*}
        \sum_{k=k_n}^{n-\ell} u_n(k) \le C x_n \e^{-x_n} \left( \sum_{k=k_n}^\infty \frac{1}{k^{3/2}} + (\log \ell) \sum_{k=\ell}^\infty \frac{1}{k^{3/2}} \right) = Cx_n \e^{-x_n}\left(\frac{1}{\sqrt{k_n}}+\frac{\log \ell}{\sqrt{\ell}}\right) = o(x_n\e^{-x_n}),
    \end{equation*}
    which concludes the proof.
\end{proof}

The following corollary of Proposition~\ref{prop:upper_bound_G} states that one can do without the dependence of~$\ell$ on~$n$, making first $n \to \infty$ and then $\ell \to \infty$.

\begin{corollary}\label{cor:upper_bound_G}
    Assume~\eqref{eq:assumption_supercriticality}, \eqref{eq:assumption_boundary_case}, and~\eqref{eq:assumption_gaussianity}.
    Then $\limsup_{\ell \to \infty} \limsup_{n \to \infty} \P(G_{n,\ell})/(x_n\e^{-x_n})=0$.
\end{corollary}

\begin{proof}
    It is a straightforward consequence of Proposition~\ref{prop:upper_bound_G} and the Lemma~\ref{lem:from_sequence_to_large_constant}.
\end{proof}

\begin{remark}
    We mention that, to prove Corollary~\ref{cor:upper_bound_G}, it is also possible to directly adapt arguments from the proof of Proposition~\ref{prop:upper_bound_G}.
\end{remark}

\section{An intermediate equivalent for the moderate deviation probability}\label{sct:equivalent_expectation}

Let $(x_n)_{n \ge 0}$ be such that $x_n \to \infty$ and $x_n = O(\sqrt{n})$ as $n \to \infty$.
Recall that $M_n^{(u)}$ is defined in~\eqref{eq:relative_max} and~$f_n(k)$ in~\eqref{eq:def_barrier}.
We set
\begin{equation*}
    \Lambda_{n,\ell} \coloneq \sum_{|u|=n-\ell} \mathds{1}_{\{\forall i \le n-\ell, V(u_i) \le f_{n-\ell}(i)+x_n\}} \mathds{1}_{\{V(u)+M_\ell^{(u)}>m_n+x_n\}}.
\end{equation*}
The main result of this section is the following proposition.

\begin{proposition}\label{prop:expectation_equivalent}
    Assume~\eqref{eq:assumption_supercriticality}, \eqref{eq:assumption_boundary_case}, \eqref{eq:assumption_gaussianity}, and~\eqref{eq:assumption_peeling_lemma}.
    As $n \to \infty$, if $\ell_n \to \infty$ and $\ell_n = o(n)$, then
    \[\P(M_n > m_n + x_n) \sim \E[\Lambda_{n,\ell_n}].\]
\end{proposition}

Let us write $\ell = \ell_n$ for short.
Note that
\begin{equation*}
    \P(M_n > m_n + x_n) \ge \P(\Lambda_{n,\ell} \ge 1) \ge \frac{\E[\Lambda_{n,\ell}]^2}{\E[\Lambda_{n,\ell}^2]},
\end{equation*}
by the Cauchy--Schwarz inequality.
Therefore, a way to obtain a lower bound for $\P(M_n > m_n + x_n)$ would be to show that $\E[\Lambda_{n,\ell}] \sim \E[\Lambda_{n,\ell}^2]$.
Unfortunately, the second moment of~$\Lambda_{n,\ell}$ is not necessarily finite under our assumptions.
For this reason, we rather consider a truncated version of~$\Lambda_{n,\ell}$.
For $u\in\T$, we denote by $\Omega(u)$  the set of children of~$u$ and write $\Omega_k = \Omega(w_k)\setminus\{w_{k+1}\}$ for short.
We will see in the proof of Lemma~\ref{lem:equivalent_Lambda_Lambda2} that an appropriate way to truncate~$\Lambda_{n,\ell}$ is to add a control on the quantity
\begin{equation*}
    \xi(u) = \log_+ \sum_{v \in \Omega(u)} (1+(V(u)-V(v))_+) \e^{V(v)-V(u)}.
\end{equation*}
In other words, we restrict~$\Lambda_{n,\ell}$ to particles that do not have ``too many children too high''.
More precisely, we consider, for $K > 0$ and $x\in\R$,
\begin{align}
    & E(x)\coloneq\{u\in\T : \forall j \le |u|, V(u_j) \le f_{|u|}(j)+x\}, \label{eq:def_E} \\
    & F_{K}(x)\coloneq\{u\in \T : \forall j< |u|, \xi(u_j) \le K+(x+f_{|u|}(j)-V(u_j))/2\},\label{eq:def_F} \\
    & \tilde{E}_K(x) \coloneq E(x)\cap F_K(x), \label{eq:def_E_tilde}
\end{align}
and
\begin{equation*}
    \Lambda_{n,\ell,K}\coloneq\sum_{|u|=n-\ell}\mathds{1}_{\{u\in\tilde{E}_K(x_n)\}} \mathds{1}_{\{V(u)+M_\ell^{(u)}>m_n+x_n\}}.
\end{equation*}

\begin{lemma}\label{lem:lower_bound_Lambda}
     Assume~\eqref{eq:assumption_supercriticality}, \eqref{eq:assumption_boundary_case}, \eqref{eq:assumption_gaussianity}, and~\eqref{eq:assumption_peeling_lemma}.
     If $\ell_n \to \infty$ and $\ell_n = o(n)$, then there exists $c > 0$ such that, for any $n \ge 1$, we have $\E[\Lambda_{n,\ell_n}] \ge c x_n \e^{-x_n}$.
\end{lemma}

\begin{proof}
    Since $\E[\Lambda_{n,\ell}] \ge \P(M_n > m_n + x_n) - \P(G_{n,\ell})$, it suffices to apply~\eqref{eq:mallein_lower_bound} and Proposition~\ref{prop:upper_bound_G}.
\end{proof}

\begin{lemma}\label{lem:peeling_lemma}
 Assume~\eqref{eq:assumption_supercriticality}, \eqref{eq:assumption_boundary_case}, \eqref{eq:assumption_gaussianity}, and~\eqref{eq:assumption_peeling_lemma}.
 If $\ell > 0$ and $\ell = o(n)$, then 
 \[\lim_{K \to \infty} \limsup_{n \to \infty} \frac{\E[\Lambda_{n,\ell}]}{\E[\Lambda_{n,\ell,K}]} = 1.\]
\end{lemma}

\begin{proof}
    Recall that $E(x)$ and $F_K(x)$ are defined in~\eqref{eq:def_E} and~\eqref{eq:def_F}.
    Using the branching property and then the probability distribution~$\bar{\P}$ and~$\hat{\P}$ introduced in Section~\ref{sct:spinal_decomposition}, we rewrite
    \begin{align*}
        \E[\Lambda_{n,\ell}] - \E[\Lambda_{n,\ell,K}]
        & = \E \left[ \sum_{|u|=n-\ell} \mathds{1}_{\{u \in E(x_n) \setminus F_K(x_n)\}} \P_{V(u)}(M_\ell > m_n+x_n) \right] \\
        & = \bar{\E} \left[ \sum_{|u|=n-\ell} \frac{\e^{V(u)}}{W_n} \e^{-V(u)} \mathds{1}_{\{u \in E(x_n) \setminus F_K(x_n)\}} \P_{V(u)}(M_\ell > m_n+x_n) \right] \\
        & = \hat{\E} \left[ \e^{-V(w_{n-\ell})} \mathds{1}_{\{w_{n-\ell} \in E(x_n) \setminus F_K(x_n)\}} \P_{V(w_{n-\ell})}(M_\ell > m_n+x_n) \right],
    \end{align*}
    by the spinal decomposition theorem (Theorem~\ref{th:spinal_decomposition}).
    By decomposing the event $\{V(w_{n-\ell}) \le f_{n-\ell}(n-\ell) + x_n\}$, we obtain
    \begin{equation*}
        \E[\Lambda_{n,\ell}] - \E[\Lambda_{n,\ell,K}] \le \sum_{i \ge 0} \e^{-x_n-m_{n-\ell}+i+1} \hat{\P} \left( \substack{w_{n-\ell}\in E(x_n) \setminus F_K(x_n), \\ x_n+m_{n-\ell}-V(w_{n-\ell}) \in [i,i+1]} \right) \P(M_\ell > m_n - m_{n-\ell} + i),
    \end{equation*}
    since $f_{n-\ell}(n-\ell) = m_{n-\ell}$.
    By Lemma~\ref{lem:ballot_additional_variables}, applied with $\xi_{j+1} = 2\xi(w_j) - 2K$,
    \begin{equation*}
        \hat{\P} \left( \substack{w_{n-\ell}\in E(x_n) \setminus F_K(x_n), \\ x_n+m_{n-\ell}-V(w_{n-\ell}) \in [i,i+1]} \right) \le C\eta(K)\frac{(x_n+1)(i+1)}{(n-\ell)^{3/2}},
    \end{equation*}
    where
    \begin{align*}
        \eta(K)
        & = \hat{\P}(\xi_1 - V(w_1) \ge 0)^{1/2} + \hat{\E} \left[ (\xi_1 - V(w_1))_+^2 \right]^{1/2} + \hat{\P}(\xi_1+1 \ge 0)^{1/2} + \hat{\E} \left[ (\xi_1+1)_+^2 \right]^{1/2} \\
        & \le \hat{\P}(2\xi(w_0)_+-V(w_1) \ge 2K)^{1/2} + \hat{\E} \left[ (2\xi(w_0)_+ - 2K - V(w_1))^2 \right]^{1/2} \\
        & \qquad + \hat{\P}(2\xi(w_0)_+ \ge 2K-1)^{1/2} + \hat{\E} \left[ (2\xi(w_0)_+ - 2K + 1)_+^2 \right]^{1/2}.
    \end{align*}
    But, by the spinal decomposition theorem and~\eqref{eq:assumption_gaussianity},
    \begin{equation*}
        \hat{\E} \left[ V(w_1)^2 \right] = \bar{\E} \left[ \sum_{|u|=1} V(u)^2 \frac{\e^{V(u)}}{W_1} \right] = \E \left[ \sum_{|u|=1} V(u)^2 \e^{V(u)} \right] < \infty,
    \end{equation*}
    where $W_n$ is defined in~\eqref{eq:additive_martingale}.
    In the same way, by the spinal decomposition theorem and~\eqref{eq:assumption_peeling_lemma} together with~\cite[Lemma~B.1]{Aidekon2013},
    \begin{equation*}
        \hat{\E} \left[ \xi(w_0)^2 \right] \le \bar{\E} \left[ \log_+^2 \sum_{|u|=1} (1+V(u)_-) \e^{V(u)} \right] = \E \left[ W_1 \log_+^2 \sum_{|u|=1} (1+V(u)_-) \e^{V(u)} \right] < \infty.
    \end{equation*}
    It follows that $\eta(K) = o_K(1)$.
    Hence,
    \begin{equation*}
        \E[\Lambda_{n,\ell}] - \E[\Lambda_{n,\ell,K}] \le x_n\e^{-x_n} o_K(1) \left( \sum_{i \ge0} (i+1)\e^{i+1} \P(M_\ell > i-1) \right).
    \end{equation*}
    By Proposition~\ref{prop:finite_integral} and Lemma~\ref{lem:negligible_integral}, the last sum in brackets is bounded by some constant.
    In view of Lemma~\ref{lem:lower_bound_Lambda}, this concludes.
\end{proof}

The following lemma shows that the second moment of~$\Lambda_{n,\ell,K}$ is asymptotically equivalent to its expectation when~$K$ is large enough.

\begin{lemma}\label{lem:equivalent_Lambda_Lambda2}
    Assume~\eqref{eq:assumption_supercriticality}, \eqref{eq:assumption_boundary_case}, \eqref{eq:assumption_gaussianity}, and~\eqref{eq:assumption_peeling_lemma}.
    If $\ell_n \to \infty$ and $\ell_n = o(n)$, then, for~$K$ large enough,
    \[\lim_{n\to\infty}\left|\frac{\E[\Lambda_{n,\ell_n,K}^2]}{\E[\Lambda_{n,\ell_n,K}]}-1\right|=0.\]
\end{lemma}

Note that, by expanding~$\Lambda_{n,\ell,K}^2$, we have $\E[\Lambda_{n,\ell,K}^2]-\E[\Lambda_{n,\ell,K}] \ge 0$.
The following lemma gives a good approximation for the upper bound of $\E[\Lambda_{n,\ell,K}^2]-\E[\Lambda_{n,\ell,K}]$, which is helpful for the proof of Lemma~\ref{lem:equivalent_Lambda_Lambda2}.

\begin{lemma}\label{lem:short}
    Assume~\eqref{eq:assumption_supercriticality} and~\eqref{eq:assumption_boundary_case}.
    If $1 \le \ell = o(n)$, then there exists $C>0$ such that, for any $K>0$ and~$n$ large enough,
    \begin{align}
        0
        & \le \E[\Lambda_{n,\ell,K}^2] - \E[\Lambda_{n,\ell,K}] \nonumber \\
        & \le C\e^K\sum_{k=0}^{n-\ell-1}\E\left[\e^{-S_{n-\ell}}P_{n,\ell}(S_{n-\ell})\psi\left(\frac{x_n+f_{n-\ell}(k)-S_k}{2}\right)\mathds{1}_{\{\forall i \le n-\ell, S_i \le x_n+f_{n-\ell}(i)\}}\right], \label{eq:short_upper_bound}
    \end{align}
    where $P_{n,\ell}(y)\coloneq\P_y(M_\ell>m_n+x_n)$ and $\psi(y)\coloneq(1+2y_+)\e^{-y}$.
\end{lemma}

\begin{proof}
    For $u,v\in\T$, we denote by $u\wedge v$ the most recent common ancestor of~$u$ and~$v$, \ie $u\wedge v \coloneq u_k$, where $k = \max\{j \ge 0 : u_j=v_j\}$.
    Recall that $\tilde{E}_{K}(x)$ is defined in~\eqref{eq:def_E_tilde}.
    We have
    \begin{equation*}
        \Lambda_{n,\ell,K}^2 = \Lambda_{n,\ell,K} + \sum_{k=0}^{n-\ell-1}\sum_{|u|=|v|=n-\ell} \mathds{1}_{\{|u \wedge v| = k\}} \mathds{1}_{\{u,v \in \tilde{E}_{K}(x_n)\}} \mathds{1}_{\{V(v)+M_\ell^{(v)}>m_n+x_n\}} \mathds{1}_{\{V(u)+M_\ell^{(u)}>m_n+x_n\}}.
    \end{equation*}
    By the branching property, conditionally on $\hF_{n-\ell}$, the random variables $M_\ell^{(u)}$, for $|u|=n-\ell$, are \iid copies of $M_\ell$.
    Therefore, recalling that~$\bar{\E}$ and~$\hat{\E}$ are defined in Section~\ref{sct:spinal_decomposition},
    \begin{align*}
        & \E \left[ \sum_{|u|=|v|=n-\ell}\mathds{1}_{\{|u\wedge v|=k\}}\mathds{1}_{\{u,v\in \tilde{E}_{K}(x_n)\}}\mathds{1}_{\{V(v)+M_\ell^{(v)}>m_n+x_n\}}\mathds{1}_{\{V(u)+M_\ell^{(u)}>m_n+x_n\}} \right] \\
        & \quad = \E \left[ \sum_{|u|=|v|=n-\ell}\mathds{1}_{\{|u\wedge v|=k\}}\mathds{1}_{\{u,v\in \tilde{E}_{K}(x_n)\}}P_{n,\ell}(V(u))P_{n,\ell}(V(v)) \right] \\
        & \quad = \bar{\E} \left[ \sum_{|u|=|v|=n-\ell}\frac{\e^{V(u)}}{W_{n-\ell}}\e^{-V(u)}\mathds{1}_{\{|u\wedge v|=k\}}\mathds{1}_{\{u,v\in \tilde{E}_{K}(x_n)\}}P_{n,\ell}(V(u))P_{n,\ell}(V(v)) \right] \\
        & \quad = \hat{\E} \left[ \e^{-V(w_{n-\ell})}P_{n,\ell}(V(w_{n-\ell}))\mathds{1}_{\{w_{n-\ell}\in \tilde{E}_K(x_n)\}}\sum_{|v|=n-\ell}\mathds{1}_{\{|w_{n-\ell}\wedge v|=k\}}\mathds{1}_{\{v\in \tilde{E}_{K}(x_n)\}}P_{n,\ell}(V(v)) \right],
    \end{align*}
    where the last equality follows from spinal decomposition theorem (Theorem~\ref{th:spinal_decomposition}).
    By definition of~$\Omega_k$,
    \begin{equation*}
        \sum_{|v|=n-\ell}\mathds{1}_{\{|w_{n-\ell}\wedge v|=k\}}\mathds{1}_{\{v\in \tilde{E}_{K}(x_n)\}}P_{n,\ell}(V(v))= \sum_{u\in\Omega_k}\sum_{|v|=n-\ell,v\succeq u} \mathds{1}_{\{v\in\tilde{E}_K(x_n)\}}P_{n,\ell}(V(v)).
    \end{equation*}
    Let $\hG_{n-\ell}\coloneq\sigma(w_j,V(w_j),(u,V(u))_{u\in\Omega_{j-1}}:1 \le j \le n-\ell)$.
    Note that, conditionally on~$\hG_{n-\ell}$, for any $u \in \Omega_k$, the family $(V(v)-V(u) : |v|=n-\ell, v \succeq u)$ under~$\hat{\P}$ has same law as $(V(u) : |u| = n-\ell-k-1)$ under~$\P$.
    Therefore,
    \begin{equation}\label{eq:short_branching_property}
        \hat{\E} \left[ \sum_{|v|=n-\ell} \mathds{1}_{\{|w_{n-\ell} \wedge v|=k\}} \mathds{1}_{\{v \in \tilde{E}_{K}(x_n)\}} P_{n,\ell}(V(v)) \middle| \hG_{n-\ell} \right] = \sum_{u\in\Omega_k} \E_{V(u)} \left[ \sum_{|v|=n-\ell-k-1} \mathds{1}_{\{v\in\tilde{E}_K(x_n)\}} P_{n,\ell}(V(v)) \right].
    \end{equation}
    For $k'=n-\ell-k-1$ and $a \in \R$,
    \begin{align*}
        \E_a \left[ \sum_{|v|=k'} \mathds{1}_{\{v\in\tilde{E}_K(x_n)\}} P_{n,\ell}(V(v)) \right]
        & \le \E_a \left[ \sum_{|v|=k'} \mathds{1}_{\{v \in E_K(x_n)\}} P_{n,\ell}(V(v)) \right] \\
        & = \E \left[ \e^{-S_{k'}} \mathds{1}_{\{\forall i \le k', S_i \le f_{n-\ell}(k+1+i)+x_n-a\}} P_{n,\ell}(S_{k'}+a) \right],
    \end{align*}
    by the many-to-one formula.
    By decomposing the event $\{S_{k'} \le f_{n-\ell}(k+1+k') + x_n - a\}$, we obtain
    \begin{multline*}
        \E \left[ \e^{-S_{k'}} \mathds{1}_{\{\forall i \le k', S_i \le f_{n-\ell}(k+1+i)+x_n-a\}} P_{n,\ell}(S_{k'}+a) \right] \\
        \le \sum_{y \ge 0} \e^{-m_{n-\ell}-x_n+a+y+1} \P \left( \substack{\forall i \le k', S_i \le f_{n-\ell}(k+1+i)+x_n-a, \\ S_{k'} - m_{n-\ell} - x_n + a \in [-y-1,-y]} \right) \P(M_\ell > m_n - m_{n-\ell} + y).
    \end{multline*}
    For convenience, we define, for $0 \le j \le n$ and $x,y,z \in \R$,\begin{equation}\label{eq:def_chi}
        \chi_{n,j}(x,y,z) \coloneq \P(\forall i < j, S_i \le f_n(i)+x, S_j-f_n(j) \in [x-y-z,x-y]).
    \end{equation}
    We write $\chi_n = \chi_{n,n}$ for short.
    Note that
    \begin{equation*}
        f_{n-\ell}(k+1+i) = m_{n-\ell} - m_{k'} + m_{k'} - m_{n-\ell-k-i} = f_{n-\ell}(k+2) + f_{k'}(i),
    \end{equation*}
    since $k' = n-\ell-k-1$.
    Thus, taking~$n$ large enough so that $m_n - m_{n-\ell} > - 1$, we have
    \begin{multline*}
        \E \left[ \e^{-S_{k'}} \mathds{1}_{\{\forall i \le k', S_i \le f_{n-\ell}(k+1+i)+x_n-a\}} P_{n,\ell}(S_{k'}+a) \right] \\
        \le \sum_{y \ge 0} \e^{-m_{n-\ell}-x_n+a+y+1} \chi_{k'}(x_n+f_{n-\ell}(k+2)-a,y,1) \P(M_\ell > y-1).
    \end{multline*}
    Applying Lemma~\ref{lem:mallein_random_walk_estimate} and then Proposition~\ref{prop:finite_integral}, we bound the above sum by
    \begin{multline*}
        C (n-\ell)^{3/2} \e^{-x_n+a} \sum_{y \ge 0} \e^{y+1} \frac{(1 + (x_n+f_{n-\ell}(k+2)-a)_+)(1+y)}{(n-\ell-k)^{3/2}} \P(M_\ell > y-1) \\
        \le C \e^{-x_n+a-f_{n-\ell}(k)} (1 + (x_n+f_{n-\ell}(k+2)-a)_+).
    \end{multline*}
    Tracing back our inequalities, we obtain that~\eqref{eq:short_branching_property} is bounded by
    \begin{multline*}
        C \e^{-x_n-f_{n-\ell}(k)} \sum_{u \in \Omega_k} \e^{V(u)} (1+(x_n+f_{n-\ell}(k+1)-V(u))_+) \\
        \le C \e^{-x_n-f_{n-\ell}(k)+V(w_k)} (1+(x_n+f_{n-\ell}(k)-V(w_k))_+) \e^{\xi(w_k)},
    \end{multline*}
    where we have used the basic inequality $1+(a+b)_+ \le (1+a_+)(1+b_+)$.
    Now, by definition~\eqref{eq:def_F}, on the event $\{w_{n-\ell}\in F_K(x_n)\}$, we have
    \begin{equation*}
        \e^{\xi(w_k)} \le \e^{K+(x_n+f_{n-\ell}(k)-V(w_k))/2}.
    \end{equation*}
    We deduce that
    \begin{equation*}
        \E[\Lambda_{n,\ell,K}^2] - \E[\Lambda_{n,\ell,K}] \le C \e^K \sum_{k=0}^{n-\ell-1} \hat{\E} \left[ \e^{-V(w_{n-\ell})} P_{n,\ell}(V(w_{n-\ell}))\mathds{1}_{\{w_{n-\ell}\in \tilde{E}_K(x_n)\}} \psi \left( \frac{x_n+f_{n-\ell}(k)-V(w_k)}{2} \right) \right].
    \end{equation*}
    Finally, the upper bound~\eqref{eq:short_upper_bound} follows from the definition of~$\tilde{E}_K(x_n)$ and the spinal decomposition theorem (Theorem~\ref{th:spinal_decomposition}).
\end{proof}

Now we prove Lemma~\ref{lem:equivalent_Lambda_Lambda2}.

\begin{proof}[Proof of Lemma~\ref{lem:equivalent_Lambda_Lambda2}]
    To simplify notations, let us abbreviate $\ell = \ell_n$.
    By Lemma~\ref{lem:short}, for~$K > 0$ and~$n$ large enough, we have
    \begin{equation*}
        \E[\Lambda_{n,\ell,K}^2] - \E[\Lambda_{n,\ell,K}] \le C\e^K\sum_{k=0}^{n-\ell-1}\E\left[\e^{-S_{n-\ell}}P_{n,\ell}(S_{n-\ell})\psi\left(\frac{x_n+f_{n-\ell}(k)-S_k}{2}\right)\mathds{1}_{\{\forall i \le n-\ell, S_i \le x_n+f_{n-\ell}(i)\}}\right],
    \end{equation*}
    where $P_{n,\ell}(y)\coloneq\P_y(M_\ell>m_n+x_n)$ and $\psi(y)\coloneq(1+2y_+)\e^{-y}$.
    By decomposing the event $\{S_{n-\ell} \le x_n+f_{n-\ell}(n-\ell)\}$, we obtain
    \begin{multline*}
        \E\left[\e^{-S_{n-\ell}}P_{n,\ell}(S_{n-\ell})\psi\left(\frac{x_n+f_{n-\ell}(k)-S_k}{2}\right)\mathds{1}_{\{\forall i \le n-\ell, S_i \le x_n+f_{n-\ell}(i)\}}\right] \\
        \le \sum_{y \ge 0} \e^{-x_n-m_{n-\ell}+y+1} \E\left[\psi\left(\frac{x_n+f_{n-\ell}(k)-S_k}{2}\right)\mathds{1}_{\left\{\substack{\forall i < n-\ell, S_i \le x_n+f_{n-\ell}(i), \\ S_{n-\ell}-x_n-m_{n-\ell}\in[-y-1,-y]}\right\}} P_{n,\ell}(x_n+m_{n-\ell}-y)\right].
    \end{multline*}
    Then, by decomposing the event $\{S_k \le x_n+f_{n-\ell}(k)\}$ and taking~$n$ large enough so that $P_{n,\ell}(x_n+m_{n-\ell}-y) \le \P(M_\ell > y-1)$, we obtain
    \begin{multline*}
        \E\left[\psi\left(\frac{x_n+f_{n-\ell}(k)-S_k}{2}\right) \mathds{1}_{\left\{\substack{\forall i < n-\ell, S_i \le x_n+f_{n-\ell}(i), \\ S_{n-\ell}-x_n-m_{n-\ell}\in[-y-1,-y]}\right\}} P_{n,\ell}(x_n+m_{n-\ell}-y)\right] \\
        \le \sum_{z \ge 0} (2+z) \e^{-z/2} \P\left(\substack{\forall i < n-\ell, S_i \le x_n+f_{n-\ell}(i), \\ S_k-x_n-f_{n-\ell}(k) \in [-z-1,-z], \\ S_{n-\ell}-x_n-m_{n-\ell} \in [-y-1,-y]}\right) \P(M_\ell > y-1).
    \end{multline*}
    By the Markov property at time~$k$,
    \begin{align*}
        \P\left(\substack{\forall i < n-\ell, S_i \le x_n+f_{n-\ell}(i), \\ S_k-x_n-f_{n-\ell}(k) \in [-z-1,-z], \\ S_{n-\ell}-x_n-m_{n-\ell} \in [-y-1,-y]}\right)
        & \le \P\left(\substack{\forall i < k, S_i \le x_n+f_{n-\ell}(i), \\ S_k-x_n-f_{n-\ell}(k) \in [-z-1,-z]}\right) \P\left(\substack{\forall i < n-\ell-k, S_i \le f_{n-\ell}(i+k)-f_{n-\ell}(k)+z+1, \\ S_{n-\ell-k}-m_{n-\ell}+f_{n-\ell}(k)-z \in [-y-1,-y+1]}\right) \\
        & \le \chi_{n-\ell,k}(x_n,z,1) \chi_{n-\ell-k}(z+1,y-1,2),
    \end{align*}
    where~$\chi_{n-\ell,k}$ and $\chi_{n-\ell-k} = \chi_{n-\ell-k,n-\ell-k}$ are defined in~\eqref{eq:def_chi}, since
    \begin{equation*}
        f_{n-\ell}(i+k) - f_{n-\ell}(k) = m_{n-\ell} - m_{n-\ell-i-k+1} - m_{n-\ell} + m_{n-\ell-k+1} \le f_{n-\ell-k}(i).
    \end{equation*}
    Tracing back our inequalities, this yields 
    \begin{equation}\label{eq:equivalent_Lambda_Lambda2}
        \E[\Lambda_{n,\ell,K}^2] - \E[\Lambda_{n,\ell,K}] \le C \e^{K-x_n} (n-\ell)^{3/2} \sum_{k=0}^{n-\ell-1} \sum_{y,z \ge 0} T_{n,k}(y,z),
    \end{equation}
    where
    \begin{equation*}
        T_{n,k}(y,z) \coloneq (1+z) \e^{y-z/2} \chi_{n-\ell,k}(x_n,z,1) \chi_{n-\ell-k}(z+1,y-1,2) \P(M_\ell > y-1).
    \end{equation*}
    
    Consider $\alpha\in(0,1)$ that we will fix later on.
    In the remainder of the proof, we treat separately the cases $k\in[0,\ell^\alpha]$, $k\in[\ell^\alpha,n-\ell-\ell^\alpha]$, and $k\in(n-\ell-\ell^\alpha,n-\ell]$.

    First, assume $k\in[\ell^{\alpha},n-\ell-\ell^\alpha]$.
    Bounding~$\chi_{n-\ell,k}(x_n,z,1)$ and~$\chi_{n-\ell-k}(z+1,y-1,2)$ via Lemma~\ref{lem:mallein_random_walk_estimate}, we obtain
    \begin{align*}
        \sum_{y,z \ge 0} T_{n,k}(y,z)
        & \le \frac{C x_n}{k^{3/2}(n-\ell-k)^{3/2}} \sum_{z \ge 0} (1+z)^3 \e^{-z/2} \sum_{y \ge 0} (1+y) \e^y \P(M_\ell > y-1) \\
        & \le \frac{Cx_n}{k^{3/2}(n-\ell-k)^{3/2}},
    \end{align*}
    where the last inequality follows from Proposition~\ref{prop:finite_integral} and Lemma~\ref{lem:negligible_integral}.
    Hence,
    \begin{align}
        (n-\ell)^{3/2} \sum_{\ell^\alpha \le k \le n-\ell-\ell^\alpha} \sum_{y,z \ge 0} T_{n,k}(y,z)
        & \le Cx_n \sum_{\ell^\alpha \le k \le n-\ell-\ell^\alpha }\frac{(n-\ell)^{3/2}}{k^{3/2}(n-\ell-k)^{3/2}} \nonumber \\
        & \le Cx_n \left( \sum_{\ell^\alpha \le k \le (n-\ell)/2} k^{-3/2}+\sum_{(n-\ell)/2 \le k \le n-\ell-\ell^\alpha} (n-\ell-k)^{-3/2} \right) \nonumber \\
        & \le Cx_n \ell^{-\alpha/2}. \label{eq:equivalent_Lambda_Lambda2_1}
    \end{align}

    Now, assume $k\in[n-\ell-\ell^{\alpha},n-\ell]$.
    Bounding~$\chi_{n-\ell,k}(x_n,z,1)$ via Lemma~\ref{lem:mallein_random_walk_estimate}, we obtain
    \begin{equation*}
        \sum_{y,z \ge 0} T_{n,k}(y,z) \le \frac{C x_n}{k^{3/2}} \sum_{y,z \ge 0} (1+z)^2 \e^{y-z/2} \chi_{n-\ell-k}(z+1,y-1,2) \P(M_\ell > y-1).
    \end{equation*}
    Consider $\beta > 0$ that will be fixed later on.
    Bounding $\chi_{n-\ell-k}(z+1,y-1,2) \le 1$ and then applying Lemma~\ref{lem:negligible_integral}, we obtain
    \begin{equation*}
        \sum_{y \ge 0} \sum_{z > \ell^\beta} T_{n,k}(y,z) \le \frac{C x_n}{k^{3/2}} \sum_{z > \ell^\beta} (1+z)^2 \e^{-z/2} \sum_{y \ge 0} \e^y \P(M_\ell > y-1) \le \frac{C x_n}{k^{3/2}} \e^{-\ell^\beta/3}.
    \end{equation*}
    Bounding $\chi_{n-\ell-k}(z+1,y-1,2) \le 1$ and then applying the upper bound~\eqref{eq:hu_upper_bound}, we also obtain
    \begin{align*}
        \sum_{0 \le y \le 2\ell^\beta} \sum_{z \le \ell^\beta} T_{n,k}(y,z)
        & \le \frac{C x_n}{k^{3/2}} \sum_{z \le \ell^\beta} (1+z)^2 \e^{-z/2} \sum_{0 \le y \le 2\ell^\beta} \e^y \P(M_\ell > y-1) \\
        & \le \frac{C x_n}{k^{3/2}} \sum_{0 \le y \le 2\ell^\beta} \frac{y-m_\ell}{\ell^{3/2}} \\
        & \le \frac{C x_n (\ell^{2\beta}+\ell^\beta\log \ell)}{k^{3/2} \ell^{3/2}}.
    \end{align*}
    When $z \le \ell^\beta$ and $y > 2\ell^\beta$, we have $z-y < -\ell^\beta$, which implies that,
    \begin{equation*}
        \chi_{n-\ell-k}(z+1,y-1,2) \le \P(S_{n-\ell-k} \le m_{n-\ell-k}+z-y+2) \le \P(S_{n-\ell-k} \le -\ell^\beta+2),
    \end{equation*}
    since $m_{n-\ell-k} \le 0$.
    Thus,
    \begin{align*}
        \sum_{y > 2\ell^\beta} \sum_{z \le \ell^\beta} T_{n,k}(y,z)
        & \le \frac{C x_n}{k^{3/2}} \P(-S_{n-\ell-k} \ge \ell^\beta-2) \sum_{z \le \ell^\beta} (1+z)^2 \e^{-z/2} \sum_{y > 2\ell^\beta} \e^y \P(M_\ell > y-1) \\
        &\le \frac{C x_n}{k^{3/2}} \frac{\E[S_{n-\ell-k}^2]}{\ell^{2\beta}} \sum_{y > 2\ell^\beta} \e^y \P(M_\ell > y-1) \\
        & \le \frac{C x_n}{k^{3/2} \ell^{2\beta-\alpha}},
    \end{align*}
    where the final inequality follows from Lemma~\ref{lem:negligible_integral}.
    Therefore, for~$\ell$ large enough, using that $\ell = o(n)$,
    \begin{equation}\label{eq:equivalent_Lambda_Lambda2_2}
        (n-\ell)^{3/2} \sum_{n-\ell-\ell^\alpha \le k \le n-\ell} \sum_{y,z \ge 0} T_{n,k}(y,z) \le Cx_n (\ell^\alpha \e^{-\ell^\beta/3} + \ell^{-3/2+2\beta+\alpha} + \ell^{2\alpha-2\beta}).
    \end{equation}
    Therefore, we will choose~$\alpha$ and~$\beta$ such that $3/2-2\beta-\alpha>0$ and $\beta > \alpha$.

    Finally, assume $k\in[0,\ell^\alpha]$.
    Bounding~$\chi_{n-\ell,k}(x_n,z,1)$ and~$\chi_{n-\ell-k}(z+1,y-1,2)$ via Lemma~\ref{lem:mallein_random_walk_estimate}, we obtain
    \begin{align*}
        \sum_{y,z \ge 0} T_{n,k}(y,z)
        & \le \frac{C (x_n \wedge \sqrt{k})}{(1+k)^{3/2}(n-\ell-k)^{3/2}} \sum_{z \ge 0} (1+z)^3 \e^{-z/2} \sum_{y \ge 0} (1+y) \e^y \P(M_\ell > y-1) \\
        & \le \frac{C (x_n \wedge \sqrt{k})}{(1+k)^{3/2}(n-\ell-k)^{3/2}},
    \end{align*}
    where the last inequality follows from Proposition~\ref{prop:finite_integral} and Lemma~\ref{lem:negligible_integral}.
    Since $\ell = o(n)$, this implies that
    \begin{align}
        (n-\ell)^{3/2} \sum_{0 \le k \le \ell^\alpha} \sum_{y,z \ge 0} T_{n,k}(y,z)
        & \le C \sum_{0 \le k \le \ell^\alpha} \frac{x_n \wedge \sqrt{k}}{(1+k)^{3/2}(n-\ell-k)^{3/2}} \nonumber \\
        & \le C \sum_{0 \le k \le x_n^2} \frac{1}{1+k} + C \sum_{x_n^2 < k \le \ell^\alpha}\frac{x_n}{k^{3/2}} \nonumber \\
        & \le C (1+\log x_n). \label{eq:equivalent_Lambda_Lambda2_3}
    \end{align}

    Inserting the bounds~\eqref{eq:equivalent_Lambda_Lambda2_1}, \eqref{eq:equivalent_Lambda_Lambda2_2}, and~\eqref{eq:equivalent_Lambda_Lambda2_3} into~\eqref{eq:equivalent_Lambda_Lambda2}, we obtain
    \begin{align*}
        \E[\Lambda_{n,\ell,K}^2] - \E[\Lambda_{n,\ell,K}]
        & \le C \e^{K-x_n} x_n (\ell^{-\alpha/2}+\ell^\alpha \e^{-\ell^\beta/3}+\ell^{-3/2+2\beta+\alpha}+\ell^{2\alpha-2\beta}+x_n^{-1}(1+\log x_n)) \\
        & \le C \e^K x_n \e^{-x_n}(\ell^{-3/16}+x_n^{-1}(1+\log x_n)),
    \end{align*}
    where we set the optimal choice $\alpha = 3/8$ and $\beta = 15/32$.
    By Lemmas~\ref{lem:lower_bound_Lambda} and~\ref{lem:peeling_lemma}, this completes the proof.
\end{proof}

\begin{proof}[Proof of Proposition~\ref{prop:expectation_equivalent}]
Recall that~$G_{n,\ell}$ is defined in~\eqref{eq:def_G}.
Note that 
\begin{equation*}
\P(M_n > m_n + x_n) \le \E[\Lambda_{n,\ell_n}] + \P(G_{n,\ell_n}),
\end{equation*}
which, by Proposition~\ref{prop:upper_bound_G} and Lemma~\ref{lem:lower_bound_Lambda}, implies that 
\begin{equation}\label{eq:expectation_equivalent_1}
\limsup_{n\to\infty}\frac{\P(M_n>m_n+x_n)}{\E[\Lambda_{n,\ell_n}]} \le 1.
\end{equation}
Besides, by Cauchy--Schwarz inequality and Lemma~\ref{lem:equivalent_Lambda_Lambda2}, we have, for $K > 0$ large enough,
\begin{equation}\label{eq:expectation_equivalent_2}
\liminf_{n\to\infty}\frac{\P(M_n > m_n + x_n)}{\E[\Lambda_{n,\ell_n}]} \ge \liminf_{n\to\infty}\frac{1}{\E[\Lambda_{n,\ell_n}]}\frac{\E[\Lambda_{n,\ell_n,K}]^2}{\E[\Lambda_{n,\ell_n,K}^2]} \ge \liminf_{n\to\infty}\frac{\E[\Lambda_{n,\ell_n,K}]}{\E[\Lambda_{n,\ell_n}]} \ge 1.
\end{equation}
where the last equality follows from Lemma~\ref{lem:peeling_lemma}.
Combining~\eqref{eq:expectation_equivalent_1} and~\eqref{eq:expectation_equivalent_2}, we complete the proof.
\end{proof}

Similarly to Corollary~\ref{cor:upper_bound_G}, the following corollary states that we can make first $n \to \infty$ and then $\ell \to \infty$.
It is a straightforward consequence of Proposition~\ref{prop:expectation_equivalent} and Lemma~\ref{lem:from_sequence_to_large_constant}.
Another way to prove it is to directly adapt the arguments from the proof of the proposition.

\begin{corollary}\label{cor:expectation_equivalent}
    Assume~\eqref{eq:assumption_supercriticality}, \eqref{eq:assumption_boundary_case}, \eqref{eq:assumption_gaussianity}, and~\eqref{eq:assumption_peeling_lemma}.
    Then,
    \begin{equation*}
        \limsup_{\ell \to \infty} \limsup_{n \to \infty} \left| \frac{\P(M_n > m_n+x_n)}{\E[\Lambda_{n,\ell}]} - 1 \right| = 0.
    \end{equation*}
\end{corollary}

\section{Asymptotics of \texorpdfstring{$\E[\Lambda_{n,\ell}]$}{E[Lambda(n,ell)]}}\label{sct:asymptotics_for_expectation}

Recall that
\begin{equation*}
    \Lambda_{n,\ell} \coloneq \sum_{|u|=n-\ell} \mathds{1}_{\{\forall i \le n-\ell, V(u_i) \le f_{n-\ell}(i)+x_n\}} \mathds{1}_{\{V(u)+M_\ell^{(u)}>m_n+x_n\}}.
\end{equation*}
In this section, we provide precise upper and lower bounds for $\E[\Lambda_{n,\ell}]$, asymptotically as $n \to \infty$ and $\ell \to \infty$.
To this end, we use the branching property and the many-to-one formula to rewrite
\begin{align}
    \E[\Lambda_{n,\ell}]
    & = \E \left[ \sum_{|u|=n-\ell}\mathds{1}_{\{\forall j \le n-\ell, V(u_j) \le f_{n-\ell}(j) + x_n\}} \P_{V(u)}(M_\ell > m_n + x_n) \right] \nonumber \\
    & = \E \left[ \e^{-S_{n-\ell}} \mathds{1}_{\{\forall j \le n-\ell, S_j \le f_{n-\ell}(j) + x_n\}} \P_{S_{n-\ell}}(M_\ell > m_n + x_n) \right]. \label{eq:decomposition_Lambda_non_lattice}
\end{align}
Next, in order to apply Lemmas~\ref{lem:doney} and~\ref{lem:pain}, we need to replace the barrier $f_{n-\ell}(j) + x_n$ with a function piecewise constant in $j \in [0,n-\ell]$.

For the upper bound, stated in Proposition~\ref{prop:upper_bound} below, we choose the ``worst'' possible constant barrier, that is~$x_n$.
The entropic repulsion phenomenon discussed in Remark~\ref{rem:entropic_repulsion} below implies that, with this new barrier, $S_{n-\ell}$ is typically at position $x_n - O(\sqrt{\ell})$, whereas with the original barrier it is typically at position $x_n - (3/2)\log(n-\ell) - O(\sqrt{\ell})$.
This gap leads us to consider $\ell = \ell_n$ such that $\log(n-\ell) = o(\sqrt{\ell})$ in order to obtain the correct asymptotic via Corollary~\ref{cor:simple_barrier_estimate}.

For the lower bound, stated in Proposition~\ref{prop:lower_bound} below, we would also like to choose the ``worst'' possible constant barrier, that is here $x_n - (3/2) \log(n-\ell)$, and apply Corollary~\ref{cor:simple_barrier_estimate} again.
However, this barrier is too far below the original barrier $f_{n-\ell}(j) + x_n \approx x_n$ for small values of~$j$, which is a problem if $\log(n-\ell) \neq o(x_n)$.
Therefore, we choose a two-level barrier instead, equal to $x_n - 1$ up to an intermediate time and then $x_n - (3/2)\log(n-\ell)$.
With such a choice, Corollary~\ref{cor:double_barrier_estimate} provides the correct asymptotic.

Since we are not able yet to show that constants~$C_{\mathrm{NL}}$ and~$C_{\mathrm{L}}$ are well defined, we state the upper and lower bounds in terms of
\begin{equation}\label{eq:C_limsup}
    \bar{C}_{\mathrm{NL}} \coloneq \frac{1}{\sigma^3} \sqrt{\frac{2}{\pi}} \limsup_{\ell\to\infty} \int_0^\infty y \e^{y} \P(M_\ell>y) \d{y} \quad \text{and} \quad \bar{C}_{\mathrm{L}} \coloneq \frac{h}{\sigma^3}\sqrt{\frac{2}{\pi}}\limsup_{\ell\to\infty}\sum_{y\in(a\ell+h\Z)\cap[0,\infty)}y\e^y\P(M_\ell>y).
\end{equation}
When the branching random walk is $(h,a)$-lattice, using Proposition~\ref{prop:restriction_lower} and Lemma~\ref{lem:negligible_integral}, we obtain
\begin{align}
    \limsup_{\ell \to \infty} \int_0^\infty y \e^y \P(M_\ell>y) \d{y}
    & = \limsup_{\ell \to \infty} \sum_{j \in (a\ell+h\Z) \cap [0,\infty)} \int_j^{j+h} y \e^y \P(M_\ell>y) \d{y} \nonumber \\
    & = \limsup_{\ell \to \infty} \sum_{j \in (a\ell+h\Z) \cap [0,\infty)} j \int_j^{j+h} \e^y \P(M_\ell>y) \d{y} \nonumber \\
    & = \limsup_{\ell \to \infty} \sum_{j \in (a\ell+h\Z) \cap [0,\infty)} j \e^j (\e^h - 1) \P(M_\ell>j) \nonumber \\
    & =\frac{\e^h-1}{h}\sigma^3\sqrt{\frac{\pi}{2}}\bar{C}_{\mathrm{L}}. \label{eq:unification_limsup}
\end{align}
Therefore, by Propositions~\ref{prop:finite_integral} and~\ref{prop:positivity}, we have $\bar{C}_{\mathrm{NL}} \in (0,\infty)$ and $\bar{C}_{\mathrm{L}} \in (0,\infty)$.

\begin{proposition}\label{prop:upper_bound}
    Assume~\eqref{eq:assumption_supercriticality}, \eqref{eq:assumption_boundary_case}, \eqref{eq:assumption_gaussianity}, \eqref{eq:assumption_peeling_lemma}, and let $(x_n)$ be any sequence such that $x_n \to \infty$ and $x_n = O(\sqrt{n})$ as $n \to \infty$.
    Let $(\ell_n)$ be a sequence of non-negative integers such that $\sqrt{\ell_n}/\log n \to \infty$ and $\ell_n = o(\sqrt{n})$ as $n \to \infty$.
    \begin{enumerate}
        \item If the branching random walk is non-lattice, then
        \begin{equation*}
            \limsup_{n \to \infty} \frac{\E[\Lambda_{n,\ell_n}]}{x_n \e^{-x_n - x_n^2/(2n\sigma^2)}} \le \bar{C}_{\mathrm{NL}}.
        \end{equation*}
        \item If the branching random walk is $(h,a)$-lattice and $x_n + m_n \in an + h\Z$, then
        \begin{equation*}
            \limsup_{n \to \infty} \frac{\E[\Lambda_{n,\ell_n}]}{x_n \e^{-x_n - x_n^2/(2n\sigma^2)}} \le \bar{C}_{\mathrm{L}}.
        \end{equation*}
    \end{enumerate}
\end{proposition}

\begin{proof}
To simplify notation, let us abbreviate $\ell = \ell_n$.

\textbf{\emph{Non-lattice case.}}
Fix $h > 0$.
By decomposing the event $\{S_{n-\ell} \le f_{n-\ell}(n-\ell) + x_n\}$ in~\eqref{eq:decomposition_Lambda_non_lattice}, we obtain
\begin{align}
    \E[\Lambda_{n,\ell}]
    & = \sum_{i \ge 0} \E \left[ \e^{-S_{n-\ell}} \mathds{1}_{\left\{ \substack{\forall j \le n-\ell, S_j \le f_{n-\ell}(j) + x_n, \\ S_{n-\ell} - m_{n-\ell} - x_n \in (-(i+1)h,-ih]} \right\}} \P_{S_{n-\ell}}(M_\ell > m_n + x_n) \right] \nonumber \\
    & \le \sum_{i \ge 0} (n-\ell)^{3/2} \e^{-x_n + (i+1)h} \P\left(\substack{\forall j \le n-\ell, S_j \le f_{n-\ell}(j) + x_n, \\ S_{n-\ell} - m_{n-\ell} - x_n \in (-(i+1)h,-ih]}\right) \P(M_\ell > (i-1)h), \label{eq:upper_bound_0}
\end{align}
for~$n$ large enough.
Now, we consider $A > 0$, $B > 0$, and we split the sum into two parts according to whether $i \in [A\sqrt{\ell}/h,B\ell/h]$ or not.
We denote by $I_1(n,h)$ the sum for $i \notin [A\sqrt{\ell}/h,B\ell/h]$ and $I_2(n,h)$ the sum for $i \in [A\sqrt{\ell}/h,B\ell/h]$.
We apply Lemma~\ref{lem:mallein_random_walk_estimate} to bound
\begin{align*}
    I_1(n,h)
    & \le C x_n \e^{-x_n} \sum_{\substack{i \ge 0, \\ i \notin [A\sqrt{\ell}/h,B\ell/h]}} \e^{(i+1)h} (1+ih)(1+h) \P(M_\ell > (i-1)h) \\
    & \le C x_n \e^{-x_n} \frac{1+h}{h} \int_{[-2h,A\sqrt{\ell}-2h] \cup [B\ell-2h,\infty)} \e^{y+3h} (1+y+2h) \P(M_\ell > y) \d{y},
\end{align*}
where the last inequality follows from monotonicity.
Note that by~\eqref{eq:hu_upper_bound},
\begin{equation*}
    \limsup_{n \to \infty} \int_{[-2h,0]} \e^{y+3h} (1+y+2h) \P(M_\ell > y) \d{y} \le \limsup_{n \to \infty} C \e^{3h} (1+2h) h \P(M_\ell > -2h) = 0.
\end{equation*}
Besides, by Propositions~\ref{prop:restriction_lower} and~\ref{prop:restriction_upper}, and Lemma~\ref{lem:negligible_integral},
\begin{equation*}
    \limsup_{B \to \infty} \limsup_{A \to 0^+} \limsup_{n \to \infty} \int_{[0,A\sqrt{\ell}-2h] \cup [B\ell-2h,\infty)} \e^{y+3h} (1+y+2h) \P(M_\ell > y) \d{y} = 0.
\end{equation*}
Therefore,
\begin{equation}\label{eq:upper_bound_1}
    \limsup_{B \to \infty} \limsup_{A \to 0^+} \limsup_{n \to \infty} \frac{I_1(n,h)}{x_n \e^{-x_n-x_n^2/(2n\sigma^2)}} = 0.
\end{equation}
For $I_2(n,h)$, replacing the barrier $f_{n-\ell}(j) + x_n$ with~$x_n$, we obtain
\begin{equation*}
    I_2(n,h) \le \sum_{i \in [A\sqrt{\ell}/h,B\ell/h]} (n-\ell)^{3/2} \e^{-x_n + (i+1)h} \tilde{\chi}_{n-\ell}(x_n,ih,h) \P(M_\ell > (i-1)h),
\end{equation*}
where
\begin{equation*}
    \tilde{\chi}_k(x,y,z) \coloneq \P(\forall j < k, S_j \le x, S_k - m_k - x \in (-y-z,-y]).
\end{equation*}
Since $\sqrt{\ell}/\log n \to \infty$ and $\ell = o(\sqrt{n})$ as $n \to \infty$, Corollary~\ref{cor:simple_barrier_estimate} yields
\begin{align*}
    I_2(n,h)
    & \le \frac{(1 + o_n(1))}{\sigma^3} \sqrt{\frac{2}{\pi}} x_n \e^{-x_n-x_n^2/(2n\sigma^2)} \sum_{i \in [A\sqrt{\ell}/h,B\ell/h]} \e^{(i+1)h} ih^2 \P(M_\ell > (i-1)h) \\
    & \le \frac{(1 + o_n(1))}{\sigma^3} \sqrt{\frac{2}{\pi}} x_n \e^{-x_n-x_n^2/(2n\sigma^2)} \int_{A\sqrt{\ell}-2h}^{B\ell+h} \e^{y+3h} (y+2h) \P(M_\ell > y) \d{y},
\end{align*}
where the last inequality follows from monotonicity.
Using Lemma~\ref{lem:negligible_integral}, we deduce that
\begin{equation}\label{eq:upper_bound_2}
    \frac{I_2(n,h)}{x_n \e^{-x_n-x_n^2/(2n\sigma^2)}} \le (1 + o_n(1)) \frac{\e^{3h}}{\sigma^3} \sqrt{\frac{2}{\pi}} \int_0^\infty y \e^y \P(M_\ell > y) \d{y} + o_n(1).
\end{equation}
Coming back to~\eqref{eq:upper_bound_0}, the asymptotics~\eqref{eq:upper_bound_1} and~\eqref{eq:upper_bound_2} imply that
\begin{equation}\label{eq:upper_bound_non_asymp}
    \frac{\E[\Lambda_{n,\ell}]}{x_n \e^{-x_n-x_n^2/(2n\sigma^2)}} \le (1 + o_n(1)) \frac{\e^{3h}}{\sigma^3} \sqrt{\frac{2}{\pi}} \int_0^\infty y \e^y \P(M_\ell > y) \d{y} + o_n(1).
\end{equation}
Letting $n \to \infty$, we obtain
\begin{equation*}
    \limsup_{n \to \infty} \frac{\E[\Lambda_{n,\ell}]}{x_n \e^{-x_n-x_n^2/(2n\sigma^2)}} \le \frac{\e^{3h}}{\sigma^3} \sqrt{\frac{2}{\pi}} \limsup_{n \to \infty} \int_0^\infty y \e^y \P(M_\ell > y) \d{y} = \e^{3h} \bar{C}_{\mathrm{NL}}.
\end{equation*}
We conclude the proof by letting $h \to 0^+$ .

\textbf{\emph{Lattice case.}} Now assume the branching random walk to be $(h,a)$-lattice.
Similarly to~\eqref{eq:upper_bound_0}, we decompose
\begin{align}
    \E[\Lambda_{n,\ell}]
    & = \sum_{\substack{k \in a\ell+h\Z \\ k \ge m_n-m_{n-\ell}}} n^{3/2} \e^{k-x_n} \P \left( \substack{\forall i < n-\ell, S_i \le f_{n-\ell}(i)+x_n, \\ S_{n-\ell} = m_n + x_n - k} \right) \P(M_\ell > k) \label{eq:decomposition_Lambda_lattice} \\
    & \le \sum_{\substack{k \in a\ell+h\Z \\ k \ge m_n-m_{n-\ell}}} n^{3/2} \e^{k-x_n} \P \left( \substack{\forall i < n-\ell, S_i \le x_n, \\ S_{n-\ell} = m_n + x_n - k} \right) \P(M_\ell > k). \nonumber
\end{align}
Now, we consider $A > 0$, $B > 0$, and we split the sum into two parts according to whether $k \in [A\sqrt{\ell},B\ell]$ or not.
Applying Corollary~\ref{cor:simple_barrier_estimate} in the first case and Lemma~\ref{lem:mallein_random_walk_estimate} in the second one, we obtain
\begin{multline*}
    \E[\Lambda_{n,\ell}] \le \frac{(1+o_n(1))}{\sigma^3} \sqrt{\frac{2}{\pi}} h x_n \e^{-x_n-x_n^2/(2n\sigma^2)} \sum_{k \in (a\ell + h\Z) \cap [A\sqrt{\ell},B\ell]} k \e^k \P(M_\ell > k) \\
    + C x_n \e^{-x_n} \sum_{k \in (a\ell + h\Z) \cap [m_n-m_{n-\ell},\infty) \setminus [A\sqrt{\ell},B\ell]} (k + \log n) \e^k \P(M_\ell > k).
\end{multline*}
Note that by~\eqref{eq:hu_upper_bound},
\begin{equation*}
    \limsup_{n \to \infty} \sum_{k \in (a\ell + h\Z) \cap [m_n-m_{n-\ell},0)} (k + \log n) \e^k \P(M_\ell > k) \le \limsup_{n \to \infty} C \frac{\log n}{\ell^{3/2}} = 0.
\end{equation*}
Besides, by Propositions~\ref{prop:restriction_lower} and~\ref{prop:restriction_upper}, and Lemma~\ref{lem:negligible_integral},
\begin{equation*}
    \limsup_{B \to \infty} \limsup_{A \to 0^+} \limsup_{n \to \infty} \sum_{k \in (a\ell + h\Z) \cap [0,\infty) \setminus [A\sqrt{\ell},B\ell]} (k + \log n) \e^k \P(M_\ell > k) = 0.
\end{equation*}
It follows that
\begin{equation*}
    \limsup_{n \to \infty} \frac{\E[\Lambda_{n,\ell}]}{x_n \e^{-x_n-x_n^2/(2n\sigma^2)}} \le \frac{h}{\sigma^3}\sqrt{\frac{2}{\pi}}\limsup_{\ell\to\infty}\sum_{k\in(a\ell+h\Z)\cap[0,\infty)}k\e^k\P(M_\ell>k) = \bar{C}_{\mathrm{L}},
\end{equation*}
which completes the proof.
\end{proof}

Let us now state the lower bound.

\begin{proposition}\label{prop:lower_bound}
    Assume~\eqref{eq:assumption_supercriticality}, \eqref{eq:assumption_boundary_case}, \eqref{eq:assumption_gaussianity}, \eqref{eq:assumption_peeling_lemma}, and let $(x_n)$ be any sequence such that $x_n \to \infty$ and $x_n = O(\sqrt{n})$ as $n \to \infty$.
    \begin{enumerate}
        \item If the branching random walk is non-lattice, then
        \begin{equation*}
            \limsup_{\ell \to \infty} \liminf_{n \to \infty} \frac{\E[\Lambda_{n,\ell}]}{x_n \e^{-x_n - x_n^2/(2n\sigma^2)}} \ge \bar{C}_{\mathrm{NL}}.
        \end{equation*}
        \item If the branching random walk is $(h,a)$-lattice and $x_n + m_n \in an + h\Z$, then
        \begin{equation*}
            \limsup_{\ell \to \infty} \liminf_{n \to \infty} \frac{\E[\Lambda_{n,\ell}]}{x_n \e^{-x_n - x_n^2/(2n\sigma^2)}} \ge \bar{C}_{\mathrm{L}}.
        \end{equation*}
    \end{enumerate}
\end{proposition}

\begin{proof}
\textbf{\emph{Non-lattice case.}}
Let $h > 0$ and $\ell \ge 1$.
Similarly to~\eqref{eq:finite_integral_fatou}, by decomposing the event $\{S_{n-\ell} \le f_{n-\ell}(n-\ell)+x_n\}$ in~\eqref{eq:decomposition_Lambda_non_lattice} and then applying Fatou's lemma, we obtain
\begin{equation*}
    \liminf_{n \to \infty} \frac{\E[\Lambda_{n,\ell}]}{x_n \e^{-x_n-x_n^2/(2n\sigma^2)}} \ge \sum_{i \ge A_\ell/h} \e^{ih} \P(M_\ell > (i+1)h) \liminf_{n \to \infty} \frac{(n-\ell)^{3/2}}{x_n \e^{-x_n^2/(2n\sigma^2)}} \P\left(\substack{\forall j \le n-\ell, S_j \le f_{n-\ell}(j) + x_n, \\ S_{n-\ell} - m_{n-\ell} - x_n \in (-(i+1)h,-ih]}\right),
\end{equation*}
where~$A_\ell$ is chosen, in view of applying Corollary~\ref{cor:double_barrier_estimate}, such that $A_\ell \to \infty$ and $A_\ell = o(\sqrt{\ell})$ as $\ell \to \infty$.
Fix $\lambda \in (0,1)$ small enough so that $\log(1-\lambda) > -2/3$.
Thus, for any $k \ge 1$ and $x,y,z \ge 0$, we have $f_k(\lfloor \lambda k \rfloor) > -1$ and then
\begin{equation*}
   \P\left(\forall j \le k, S_j \le f_k(j) + x, \\ S_{k} - m_k - x \in (-y-z,-z]\right) \le \tilde{\chi}_k^{(\lambda)}(x,y,z),
\end{equation*}
where
\begin{equation*}
    \tilde{\chi}_k^{(\lambda)}(x,y,z) \coloneq \P(\forall j \le \lfloor \lambda k \rfloor, S_j \le x - 1, \forall \lfloor \lambda k \rfloor < j' < k, S_{j'} \le m_k + x, S_k-m_k - x \in (-y-z,-y]).
\end{equation*}
Therefore,
\begin{align*}
    \liminf_{n \to \infty} \frac{\E[\Lambda_{n,\ell}]}{x_n \e^{-x_n-x_n^2/(2n\sigma^2)}}
    & \ge \sum_{i \ge A_\ell/h} \e^{ih} \P(M_\ell > (i+1)h) \liminf_{n \to \infty} \frac{(n-\ell)^{3/2}}{x_n \e^{-x_n^2/(2n\sigma^2)}} \tilde{\chi}_{n-\ell}^{(\lambda)}(x_n,ih,h) \\
    & \ge \frac{(1 + o_\ell(1))}{\sigma^3} \sqrt{\frac{2}{\pi}} \sum_{i \ge A_\ell/h} \e^{ih} \P(M_\ell > (i+1)h) ih^2,
\end{align*}
by Corollary~\ref{cor:double_barrier_estimate}.
Letting $\ell \to \infty$, it follows from~\eqref{eq:finite_integral_restriction} and~\eqref{eq:finite_integral_limsup} that
\begin{equation*}
    \limsup_{\ell \to \infty} \liminf_{n \to \infty} \frac{\E[\Lambda_{n,\ell}]}{x_n \e^{-x_n-x_n^2/(2n\sigma^2)}} \ge \frac{\e^{-2h}}{\sigma^3} \sqrt{\frac{2}{\pi}} \limsup_{\ell \to \infty} \int_0^\infty y \e^y \P(M_\ell > y) \d{y} = \e^{-2h} \bar{C}_{\mathrm{NL}}.
\end{equation*}
We conclude the proof in the non-lattice case by letting $h \to 0^+$.

\textbf{\emph{Lattice case.}} Now assume the branching random walk to be $(h,a)$-lattice.
In view of the decomposition~\eqref{eq:decomposition_Lambda_lattice}, Fatou's lemma yields
\begin{equation*}
    \liminf_{n \to \infty} \frac{\E[\Lambda_{n,\ell}]}{x_n \e^{-x_n-x_n^2/(2n\sigma^2)}} \ge \sum_{k \in (a\ell+h\Z) \cap [A_\ell,\infty)} \e^k \P(M_\ell > k) \liminf_{n \to \infty} \frac{\P \left( \substack{\forall i < n-\ell, S_i \le f_{n-\ell}(i) + x_n, \\ S_{n-\ell} = m_{n-\ell} + x_n - k} \right)}{n^{-3/2} x_n \e^{-x_n^2/(2n\sigma^2)}},
\end{equation*}
where~$A_\ell$ is such that $A_\ell \to \infty$ and $A_\ell = o(\sqrt{\ell})$ as $\ell \to \infty$.
Fix $\lambda \in (0,1)$ such that $\log(1-\lambda) > -2/3$, we have
\begin{equation*}
    \P \left( \substack{\forall i < n-\ell, S_i \le f_{n-\ell}(i) + x_n, \\ S_{n-\ell} = m_{n-\ell} + x_n - k} \right) \ge \P \left( \substack{\forall i < \lfloor \lambda(n-\ell) \rfloor, S_i \le x_n-1, \\ \forall \lfloor \lambda(n-\ell) \rfloor < j < n-\ell, S_j \le m_{n-\ell} + x_n, \\ S_{n-\ell} = m_{n-\ell} + x_n - k} \right).
\end{equation*}
By Corollary~\ref{cor:double_barrier_estimate},
\begin{equation*}
    \liminf_{n \to \infty} \frac{\E[\Lambda_{n,\ell}]}{x_n \e^{-x_n-x_n^2/(2n\sigma^2)}} \ge \frac{(1 + o_\ell(1))h}{\sigma^3} \sqrt{\frac{2}{\pi}} \sum_{k \in (a\ell+h\Z) \cap [A_\ell,\infty)} k \e^k \P(M_\ell > k).
\end{equation*}
Letting $\ell \to \infty$, it follows from Proposition~\ref{prop:restriction_lower} that
\begin{equation*}
    \limsup_{\ell \to \infty} \liminf_{n \to \infty} \frac{\E[\Lambda_{n,\ell}]}{x_n \e^{-x_n-x_n^2/(2n\sigma^2)}} \ge \frac{h}{\sigma^3} \sqrt{\frac{2}{\pi}} \limsup_{\ell \to \infty} \sum_{k \in (a\ell+h\Z) \cap [0_\ell,\infty)} k \e^k \P(M_\ell > k) = \bar{C}_{\mathrm{L}},
\end{equation*}
which concludes the proof.
\end{proof}

\section{Proof of Theorem~\ref{th:moderate_deviation}}\label{sct:proof_of_theorem}

We are now able to prove our main result Theorem~\ref{th:moderate_deviation}.
At first, we show that the equivalents~\eqref{eq:moderate_deviation_nl} and~\eqref{eq:moderate_deviation_l} hold by replacing~$C_{\mathrm{NL}}$ and~$C_{\mathrm{L}}$ with~$\bar{C}_{\mathrm{NL}}$ and~$\bar{C}_{\mathrm{L}}$, defined in~\eqref{eq:C_limsup}.
Once we know such an asymptotic, we deduce that the former constants are well defined and, subsequently, equal to the latter constants.

\begin{proof}[Proof of Theorem~\ref{th:moderate_deviation}]
    First assume the branching random walk to be non-lattice and let $(x_n)$ be such that $x_n \to \infty$ and $x_n = O(\sqrt{n})$.
    By Propositions~\ref{prop:expectation_equivalent} and~\ref{prop:upper_bound}, choosing~$\ell_n$ such that $\sqrt{\ell_n}/\log n \to \infty$ and $\ell_n = o(\sqrt{n})$, we obtain
    \begin{equation*}
        \limsup_{n \to \infty} \frac{\P(M_n > m_n + x_n)}{x_n \e^{-x_n - x_n^2/(2n\sigma^2)}} = \limsup_{n \to \infty} \frac{\E[\Lambda_{n,\ell_n}]}{x_n \e^{-x_n - x_n^2/(2n\sigma^2)}} \le \bar{C}_{\mathrm{NL}}.
    \end{equation*}
    For the lower bound, we consider~$\ell$ constant in~$n$ and let $\ell \to \infty$ after $n \to \infty$.
    By Corollary~\ref{cor:expectation_equivalent} and Proposition~\ref{prop:lower_bound},
    \begin{equation*}
        \liminf_{n\to\infty} \frac{\P(M_n>m_n+x_n)}{x_n \e^{-x_n-x_n^2/(2n\sigma^2)}} = \limsup_{\ell\to\infty} \liminf_{n\to\infty} \frac{\E[\Lambda_{n,\ell}]}{x_n \e^{-x_n-x_n^2/(2n\sigma^2)}} \ge \bar{C}_{\mathrm{NL}}.
    \end{equation*}
    Besides, the constant~$\bar{C}_{\mathrm{NL}}$ is in $(0,\infty)$, by Propositions~\ref{prop:finite_integral} and~\ref{prop:positivity}.
    Hence,
    \begin{equation}\label{eq:close_to_theorem}
        \P(M_n > m_n + x_n) \sim \bar{C}_{\mathrm{NL}} x_n \e^{-x_n - x_n^2/(2n\sigma^2)}.
    \end{equation}
    Now, as above, consider $\ell_n$ such that $\sqrt{\ell_n}/\log n \to \infty$ and $\ell_n = o(\sqrt{n})$, and further assume that it takes all integer values (\eg set $\ell_n = \lfloor n^{1/3} \rfloor$).
    By~\eqref{eq:close_to_theorem} and~Proposition~\ref{prop:expectation_equivalent},
    \begin{equation*}
        \frac{1}{\sigma^3} \sqrt{\frac{2}{\pi}} \limsup_{\ell\to\infty} \int_0^\infty y \e^y\P(M_\ell>y)\d{y} = \lim_{n\to\infty} \frac{\P(M_n>m_n+x_n)}{x_n \e^{-x_n-x_n^2/(2n\sigma^2)}} = \lim_{n\to\infty} \frac{\E[\Lambda_{n,\ell_n}]}{x_n \e^{-x_n-x_n^2/(2n\sigma^2)}}.
    \end{equation*}
    Therefore, taking the $\liminf$ in~\eqref{eq:upper_bound_non_asymp} (recall that, in this context, we wrote $\ell = \ell_n$ for short), we obtain
    \begin{align*}
        \frac{1}{\sigma^3} \sqrt{\frac{2}{\pi}} \limsup_{\ell\to\infty} \int_0^\infty y \e^y\P(M_\ell>y)\d{y}
        & \le \frac{1}{\sigma^3} \sqrt{\frac{2}{\pi}} \liminf_{n\to\infty} \int_0^\infty y \e^y\P(M_{\ell_n}>y)\d{y} \\
        & = \frac{1}{\sigma^3} \sqrt{\frac{2}{\pi}} \liminf_{\ell\to\infty} \int_0^\infty y \e^y\P(M_\ell>y)\d{y}.
    \end{align*}
    Thus, the limit exists, \ie~$C_{\mathrm{NL}}$ is well defined and is equal to~$\bar{C}_{\mathrm{NL}}$.
    This, together with~\eqref{eq:close_to_theorem}, concludes the proof of~\eqref{eq:moderate_deviation_nl}.
    
    Regarding the lattice case, the same argument shows that, for $(x_n)$ such that $x_n \to \infty$, $x_n = O(\sqrt{n})$, and $x_n + m_n \in an+h\Z$, we have
    \begin{equation*}
        \P(M_n > m_n + x_n) \sim \bar{C}_{\mathrm{L}} x_n \e^{-x_n - x_n^2/(2n\sigma^2)},
    \end{equation*}
    and
    \begin{equation*}
        \frac{h}{\sigma^3}\sqrt{\frac{2}{\pi}}\limsup_{\ell\to\infty}\sum_{y\in(a\ell+h\Z)\cap[0,\infty)}y\e^y\P(M_\ell>y) \le \frac{h}{\sigma^3}\sqrt{\frac{2}{\pi}}\liminf_{\ell\to\infty}\sum_{y\in(a\ell+h\Z)\cap[0,\infty)}y\e^y\P(M_\ell>y),
    \end{equation*}
    which concludes the proof of~\eqref{eq:moderate_deviation_l} and of the theorem.
\end{proof}

\section{Applications}\label{sct:applications}

In this section, we present two applications of our main result Theorem~\ref{th:moderate_deviation}.
The first is a refinement of the study of the integral defining~$C^*$ in~\eqref{eq:unification}.
In Section~\ref{sct:integral}, we showed that the domain of integration can be restricted to $[\vep\sqrt{\ell},B\ell]$ for~$\vep$ small and~$B$ large.
Here, we show that it can even be restricted to a $\sqrt{\ell}$-neighborhood.
The second application concerns the two-speed branching random walk in the so-called \emph{mean regime}.
This is a more general model than the branching random walk, for which we prove analogs of~\eqref{eq:aidekon} and~\eqref{eq:bramson_ding_zeitouni}.

\subsection{Main support of the integral defining~\texorpdfstring{$C^*$}{C*}}\label{sct:integral_2}

Here is the main result of this section.

\begin{proposition}\label{prop:domain}
    Assume~\eqref{eq:assumption_supercriticality}, \eqref{eq:assumption_boundary_case}, \eqref{eq:assumption_gaussianity}, and~\eqref{eq:assumption_peeling_lemma}.
    For any $\vep > 0$, there exist $A, B \in (0,\infty)$ such that
    \begin{equation*}
        \limsup_{\ell\to\infty} \int_{[0,\infty)\setminus[A\sqrt{\ell},B\sqrt{\ell}]}  y \e^y \P(M_\ell>y) \d y \le \vep.
    \end{equation*}
\end{proposition}

\begin{proof}
Let $\vep \in (0,1)$.
By Proposition~\ref{prop:restriction_lower}, there exists $A > 0$ small enough so that
\begin{equation*}
    \limsup_{\ell\to\infty} \int_0^{A\sqrt{\ell}}  y \e^y \P(M_\ell>y) \d y \le \vep.
\end{equation*}
In view of~\eqref{eq:unification}, it remains to show that
\begin{equation}\label{eq:domain}
    \liminf_{B\to\infty} \liminf_{\ell\to\infty} \int_0^{B\sqrt{\ell}}y\e^{y}\P(M_\ell>y)\d{y} \ge \lim_{\ell \to \infty} \int_0^\infty y \e^y \P(M_\ell>y) \d{y} \eqcolon \sqrt{\frac{\pi}{2}}\sigma^3 C^*.
\end{equation}
Let us first assume the branching random walk to be non-lattice.
Applying Corollary~\ref{cor:uniform_convergence} with $x = y - m_\ell$, we obtain that, for~$\ell$ large enough,
\begin{equation*}
    (1-\vep)C^*\ell^{-3/2}y\e^{-y-y^2/(2\ell\sigma^2)} \le \P(M_\ell>y) \le (1+\vep)C^*\ell^{-3/2}y\e^{-y-y^2/(2\ell\sigma^2)},
\end{equation*}
uniformly in $y\in[0,B\sqrt{\ell}]$.
In particular,
\begin{equation*}
    \int_0^{B\sqrt{\ell}}y\e^{y}\P(M_\ell>y)\d{y} \ge (1-\vep)C^*\ell^{-3/2}\int_0^{B\sqrt{\ell}}y^2\e^{-y^2/(2\ell\sigma^2)}\d{y} = (1-\vep)C^*\int_0^Bz^2\e^{-z^2/(2\sigma^2)}\d{z}.
\end{equation*}
Letting $\ell \to \infty$ and then $B \to \infty$, we deduce~\eqref{eq:domain}.

Now assume the branching random walk to be $(h,a)$-lattice.
By Corollary~\ref{cor:uniform_convergence} and~\eqref{eq:unification}, for~$\ell$ large enough,
\begin{equation*}
    \P(M_\ell>y) \ge (1-\vep)\frac{h}{\e^h-1}C^*\ell^{-3/2}y\e^{-y-y^2/(2\ell\sigma^2)},
\end{equation*}
uniformly in $y\in[0,B\sqrt{\ell}]\cap(a\ell+h\Z)$.
Therefore,
\begin{align}
    \int_0^{B\sqrt{\ell}}y\e^{y}\P(M_\ell>y)\d{y}
    & \ge \sum_{y \in [0,B\sqrt{\ell}-h]\cap(a\ell+h\Z)} \P(M_\ell>y)\int_y^{y+h}x\e^x\d{x} \nonumber \\
    & \ge (1-\vep)\frac{h}{\e^h-1}C^*\ell^{-3/2} \sum_{y \in [0,B\sqrt{\ell}-h]\cap(a\ell+h\Z)} y\e^{-y-y^2/(2\ell\sigma^2)}\int_y^{y+h}x\e^x\d{x}. \label{eq:domain_lower_bound}
\end{align}
By monotonicity,
\begin{align*}
    \sum_{y \in [0,B\sqrt{\ell}-h]\cap(a\ell+h\Z)} y\e^{-y-y^2/(2\ell\sigma^2)}\int_y^{y+h}x\e^x\d{x}
    & \ge \sum_{y \in [0,B\sqrt{\ell}-h]\cap(a\ell+h\Z)} y^2\e^{-y-y^2/(2\ell\sigma^2)}\int_y^{y+h}\e^x\d{x} \\
    & = (\e^h-1) \sum_{y \in [0,B\sqrt{\ell}-h]\cap(a\ell+h\Z)} y^2\e^{-y^2/(2\ell\sigma^2)} \\
    & \ge \frac{\e^h-1}{h} \int_h^{B\sqrt{\ell}-h} (y-h)^2\e^{-y^2/(2\ell\sigma^2)}\d{y}.
\end{align*}
Inserting this lower bound in~\eqref{eq:domain_lower_bound} and performing the change of variable $z = y/\sqrt{\ell}$, we obtain
\begin{equation*}
    \int_0^{B\sqrt{\ell}}y\e^{y}\P(M_\ell>y)\d{y} \ge (1-\vep) C^* \int_{h/\sqrt{\ell}}^{B-h/\sqrt{\ell}} z^2 \e^{-z^2/(2\sigma^2)} \d{z} + o_\ell(1). 
\end{equation*}
Letting $\ell \to \infty$ and then $B \to \infty$, we deduce~\eqref{eq:domain}.
\end{proof}

\begin{remark}\label{rem:entropic_repulsion}
    A consequence of Proposition~\ref{prop:domain} is that, for any $\vep,t\in(0,1)$ and $x\in\R$, there exist $0<A<B<\infty$ such that 
    \begin{equation}\label{eq:trajectory_of_particle}
        \limsup_{n \to \infty} \P(\exists |u|=n, V(u_{\lfloor t n \rfloor}) \notin [-B\sqrt{n},
        -A\sqrt{n}], V(u) > m_n+x) \le \vep.
    \end{equation}
    We will prove it in Lemma~\ref{lem:two_speed_trajectory} below for a more general model, namely the two-speed branching random walk.
    This result shows that, with high probability, the ancestor at time $\lfloor t n\rfloor$ of particles near~$m_n$ at time~$n$ has a position of order $-\sqrt{n}$.
    This is in line with~\cite{Chen2015b}, where Chen proved that the rescaled trajectory leading to the maximal position at time~$n$ converges in law to a Brownian excursion.
    Similarly, Chen, Madaule, and Mallein~\cite{ChenMadauleMallein2019} proved that the rescaled trajectory of a particle selected according to a supercritical Gibbs measure at time~$n$ (whose position is typically at distance~$O(1)$ from~$m_n$) also converges in law to a Brownian excursion.
    Subsequently, the result~\eqref{eq:trajectory_of_particle} can be interpreted as the ``entropic repulsion'' phenomenon: a Brownian bridge of length~$n$ conditioned to stay below a line for most of its lifespan, actually stays at a distance of order $\min(\sqrt{k},\sqrt{n-k})$ below this line at each time $k \in [O(1),n-O(1)]$.
    In the framework of branching Brownian motion, this property was already present in Bramson's work~\cite{Bramson1978} and it is precisely formulated in~\cite[Theorem~2.3]{ArguinBovierKistler2011}.
\end{remark}

\subsection{Convergence of the maximum of a two-speed branching random walk}\label{sct:two_speed}

As an application of Theorem~\ref{th:moderate_deviation}, we prove the analog of~\eqref{eq:aidekon} and~\eqref{eq:bramson_ding_zeitouni} for a two-speed branching random walk, which was conjectured in~\cite{Luo2025+} for the non-lattice case.

Given $t\in(0,1)$ and $n\in\N$, a two-speed branching random walk on~$\R$ is a time-inhomogeneous branching random walk, \ie there is a change of reproduction law at time $\lfloor t n\rfloor$.
More precisely, there is one particle located at~$0$ at time~$0$, and each particle reproduces according to the law of a point process~$\cL_1$ independently from time $0$ to $\lfloor t n\rfloor - 1$, and reproduces according to the law of another point process~$\cL_2$ independently between $\lfloor t n\rfloor$ and~$n-1$.
In particular, if the law of~$\cL_1$ is equal to that of~$\cL_2$, then the two-speed branching random walk reduces to a time-homogeneous branching random walk.

We are interested in the asymptotic behavior of the maximal displacement $M_n^{(n)}$ of particles at time~$n$.
Fang and Zeitouni~\cite{FangZeitouni2012} showed the tightness of the centered $M_n^{(n)}$ when the particle displacements are Gaussian.
Furthermore, Mallein~\cite{Mallein2015} extended the tightness to a time-inhomogeneous branching random walk under general assumptions.
The weak convergence of the centered maximum has been proved by Bovier and Hartung~\cite{BovierHartung2014} for a two-speed branching Brownian motion and by Luo~\cite{Luo2025+} for a two-speed branching random walk in two different regimes, determined by a relation between~$\cL_1$ and~$\cL_2$, referred to as the \emph{slow} and \emph{fast regimes}.
The former corresponds to the case where the maximal particle at time~$n$ descends from one of the maximal particles at time $\lfloor tn \rfloor$, the latter corresponds to the case where the optimal path is at a distance of order~$n$ below the maximum at time $\lfloor tn \rfloor$ (see \eg~\cite{Mallein2015}).

In this section, we use Theorem~\ref{th:moderate_deviation} to prove that, in the \emph{mean regime}, \ie the frontier between the fast and slow regimes, the centered maximum converges in law to a standard Gumbel distribution with an independent random shift.
Before stating the main result, we introduce some notation.
For $i\in\{1,2\}$, let $(V_i(u):u\in\T)$ be a branching random walk with reproduction law~$\cL_i$, $Z_{i,\infty}$ be the limit of the associated derivative martingale~\eqref{eq:derivative_martingale}, $M_{i,n}\coloneq\max_{|u|=n}V_{i}(u)$, and
\begin{equation*}
    \sigma_i^2 \coloneq \E \left[ \sum_{|u|=1} V_i(u)^2 \e^{V_i(u)} \right].
\end{equation*}
Recall that $m_n=-(3/2)\log n$.
In the remainder of this section, we aim to prove the following theorem.

\begin{theorem}\label{th:two_speed}
    Assume that $(V_1(u):u\in\T)$ and $(V_2(u):u\in\T)$ satisfy the assumptions~\eqref{eq:assumption_supercriticality}, \eqref{eq:assumption_boundary_case}, \eqref{eq:assumption_gaussianity}, \eqref{eq:assumption_peeling_lemma}, and fix $t \in (0,1)$.
    \begin{enumerate}
        \item If $(V_2(u):u\in\T)$ is non-lattice, then
        \begin{equation}\label{eq:two_speed_non_lattice}
            \lim_{n \to \infty} \sup_{x \in \R} \left| \P(M_n^{(n)} \le m_n + x) - \E \left[ \exp \left( -\frac{C_{2,\mathrm{NL}} Z_{1,\infty} \e^{-x}}{(1-t+t\sigma_1^2/\sigma_2^2)^{3/2}} \right) \right] \right| = 0,
        \end{equation}
        where
        \begin{equation*}
            C_{2,\mathrm{NL}} \coloneq \frac{1}{\sigma_2^3} \sqrt{\frac{2}{\pi}} \lim_{\ell\to\infty} \int_0^\infty y \e^{y} \P(M_{2,\ell}>y) \d{y} \in (0,\infty).
        \end{equation*}
        \item If $(V_1(u):u\in\T)$ is $(h_1,a_1)$-lattice, $(V_2(u):u\in\T)$ is $(h_2,a_2)$-lattice, and $h_1\Z \subset h_2\Z$, then
        \begin{equation}\label{eq:two_speed_lattice}
            \lim_{n \to \infty} \sup_{x \in A_n} \left| \P(M_n^{(n)} \le m_n + x) - \E \left[ \exp \left( -\frac{C_{2,\mathrm{L}} Z_{1,\infty} \e^{-x}}{(1-t+t\sigma_1^2/\sigma_2^2)^{3/2}} \right) \right] \right| = 0,
        \end{equation}
        where
        \begin{equation*}
            A_n = \{x : m_n + x \in a_1 \lfloor tn \rfloor + a_2(n - \lfloor tn \rfloor) + h_2\Z\},
        \end{equation*}
        and
        \begin{equation*}
            C_{2,\mathrm{L}} \coloneq \frac{h_2}{\sigma_2^3}\sqrt{\frac{2}{\pi}}\lim_{\ell\to\infty}\sum_{y\in(a_2\ell+h_2\Z)\cap[0,\infty)}y\e^y\P(M_{2,\ell}>y)\in(0,\infty).
        \end{equation*}
    \end{enumerate}
\end{theorem}

\begin{remark}
    The case where $(V_1(u):u\in\T)$ is $(h_1,a_1)$-lattice, $(V_2(u):u\in\T)$ is $(h_2,a_2)$-lattice, and $h_1\Z \not\subset h_2\Z$ is more intricate and we do not treat it here.
    The same holds for the case where $(V_1(u):u\in\T)$ is non-lattice and $(V_2(u):u\in\T)$ is lattice.
    In view of our proof, the issue in such a case is that, for $x \in \R$ and $n \ge 1$, the family $\{x-V_1(u) : |u|=t_n\}$ is not included in a lattice of span~$h_2$.
    This prevents us to have a uniform control of the quantities $\P_{V_1(u)}(M_{2,n-t_n} > m_n + x)$, for $|u|=t_n$, via Corollary~\ref{cor:uniform_convergence}.
\end{remark}

Before proving Theorem~\ref{th:two_speed}, we introduce a helpful lemma, which is a slight extension of~\eqref{eq:trajectory_of_particle}.
For $n \in \N$ and $u\in\T$ such that $0 \le |u| \le n$, we write $V^{(n)}(u)$ the position of~$u$ and $t_n \coloneq \lfloor tn \rfloor$.

\begin{lemma}\label{lem:two_speed_trajectory}
Assume that $(V_1(u):u\in\T)$ and $(V_2(u):u\in\T)$ satisfy the assumptions~\eqref{eq:assumption_supercriticality}, \eqref{eq:assumption_boundary_case}, \eqref{eq:assumption_gaussianity}, \eqref{eq:assumption_peeling_lemma}, and fix $t \in (0,1)$.
Then, for any $\vep>0$ and $x\in\R$, there exist $0<A<B<\infty$ such that 
\[\limsup_{n\to\infty}\P(\exists|u|=n,V^{(n)}(u_{t_n})\notin[-B\sqrt{n},-A\sqrt{n}],V^{(n)}(u)>m_n+x)\le\vep.\]
\end{lemma}

\begin{proof}
One can show that, for $\gamma>0$, we have $\P(\exists u\in\T:V(u)>\gamma) \le \e^{-\gamma}$ (see \eg~\cite[Lemma~3.2]{Chen2015a}).
Therefore,
\begin{align*}
    \P(\exists|u|=n,V^{(n)}(u_{t_n})\notin[-B\sqrt{n},-A\sqrt{n}],V^{(n)}(u)>m_n+x)
    & \le \e^{-\gamma} + \P \left( \substack{\exists|u|=n, \forall k \le t_n, V^{(n)}(u_k) \le \gamma, \\ V^{(n)}(u_{t_n}) \in (-\infty,-B\sqrt{n}]\cup[-A\sqrt{n},\gamma], \\ V^{(n)}(u)>m_n+x} \right) \\
    & \le \e^{-\gamma} + \P \left( \substack{\exists|u|=t_n, \forall k \le t_n, V^{(n)}(u_k) \le \gamma, \\ V^{(n)}(u)\in(-\infty,-B\sqrt{n}]\cup[-A\sqrt{n},\gamma], \\ \max_{v\succeq u,|v|=n} V^{(n)}(v)>m_n+x} \right).
\end{align*}
By Markov's inequality and the Markov property at time~$t_n$, the last probability term is bounded by
\begin{multline*}
    \E \left[ \sum_{|u|=t_n}\mathds{1}_{\left\{ \substack{\forall k \le t_n, V_1(u_k) \le \gamma, \\ V_1(u)\in(-\infty,-B\sqrt{n}]\cup[-A\sqrt{n},\gamma]} \right\}} \P_{V_1(u)}(M_{2,n-t_n}>m_n+x) \right] \\
    = \E \left[ \e^{-S_{1,t_n}} \mathds{1}_{\left\{ \substack{\forall k \le t_n, S_{1,k} \le \gamma, \\ S_{1,t_n}\in(-\infty,-B\sqrt{n}]\cup[-A\sqrt{n},\gamma]} \right\}} \P_{S_{1,t_n}}(M_{2,n-t_n}>m_n+x) \right],
\end{multline*}
where the equality follows from the many-to-one formula and $(S_{1,n})_{n \ge 0}$ is a random walk associated with the branching random walk $(V_1(u):u\in \T)$.
By decomposing the event $\{S_{1,t_n}\in(-\infty,-B\sqrt{n}]\cup[-A\sqrt{n},\gamma]\}$, we obtain
\begin{align*}
    & \E \left[ \e^{-S_{1,t_n}}\mathds{1}_{\left\{ \substack{\forall k \le t_n, S_{1,k} \le \gamma, \\ S_{1,t_n}\in(-\infty,-B\sqrt{n}]\cup[-A\sqrt{n},\gamma]} \right\}}\P_{S_{1,t_n}}(M_{2,n-t_n}>m_n+x) \right] \\
    & \quad \le \sum_{\substack{i \ge \lfloor -\gamma \rfloor, \\ i\notin[A\sqrt{n},B\sqrt{n}-1]}}\e^{i+1} \P \left( \substack{\forall k \le t_n, S_{1,k} \le \gamma, \\ S_{1,t_n}\in[-i-1,-i]} \right) \P(M_{2,n-t_n}>m_n+x+i) \\
    & \quad \le \frac{C}{n^{3/2}}\sum_{\substack{i \ge \lfloor -\gamma \rfloor, \\ i\notin[A\sqrt{n},B\sqrt{n}-1]}}(i+\gamma)\e^{i+1}\P(M_{2,n-t_n}>m_n+x+i) \\
    & \quad \le \frac{C}{n^{3/2}}\int_{[-\gamma-1,A\sqrt{n}+1]\cup[B\sqrt{n}-1,\infty)} (y+\gamma)\e^y\P(M_{2,{n-t_n}}>m_n+x+y-1)\d y,
\end{align*}
where the second inequality follows from Lemma~\ref{lem:mallein_random_walk_estimate}.
With the change of variable $z=m_n+x+y-1$, we rewrite the above bound as
\[C\e^x\int_{([-\gamma-1,A\sqrt{n}+1]\cup[B\sqrt{n}-1,\infty))+m_n+x-1} (z-m_n-x+\gamma+1)\e^z\P(M_{2,n-t_n}>z)\d z.\]
According to Lemma~\ref{lem:negligible_integral}, we have
\begin{equation}\label{eq:two_speed_trajectory_1}
    \limsup_{n\to\infty}\int_{[0,\infty)}(-m_n-x+\gamma+1)\e^z\P(M_{2,n-t_n}>z)\d z=0.
\end{equation}
Concerning the negative integration domain, we can bound
\begin{multline}\label{eq:two_speed_trajectory_2}
    0 \le \int_{[-\gamma+m_n+x-2,0)} (z-m_n-x+\gamma+1)\e^z\P(M_{2,n-t_n}>z)\d z \\
    \le C \log n \int_{-\infty}^{m_n/2} \e^z \d{z} + C \log n \int_{m_n/2}^0 \e^z \P(M_{2,n-t_n}>m_{n-t_n}-m_n+\frac{3}{2}\log(1-t)+z) \d{z},
\end{multline}
since $m_{n-t_n} \coloneq -\frac{3}{2}\log(n - \lfloor tn \rfloor) \le m_n - \frac{3}{2}\log(1-t)$.
In the right-hand side of~\eqref{eq:two_speed_trajectory_2}, computing the first term and bounding the second term via~\eqref{eq:hu_upper_bound}, we obtain that~\eqref{eq:two_speed_trajectory_2} converges to~$0$ as $n \to \infty$.
We deduce from~\eqref{eq:two_speed_trajectory_1} and~\eqref{eq:two_speed_trajectory_2} that
\begin{multline*}
\limsup_{n\to\infty}\P(\exists|u|=n,V^{(n)}(u_{t_n})\notin[-B\sqrt{n},-A\sqrt{n}],V^{(n)}(u)>m_n+x) \\
\le \e^{-\gamma}+C\limsup_{n\to\infty}\int_{[0,A\sqrt{n}+m_n+x]\cup[B\sqrt{n}+m_n+x-2,\infty)} z\e^z\P(M_{2,n-t_n}>z)\d z.
\end{multline*}
By Proposition~\ref{prop:domain}, for any $\vep \in (0,1)$, there exist~$A$ small enough and~$B$ large enough such that the above quantity is bounded by $\e^{-\gamma} + \vep$.
Letting $\gamma\to \infty$), this completes the proof.
\end{proof}

Now we prove Theorem~\ref{th:two_speed}.

\begin{proof}[Proof of Theorem~\ref{th:two_speed}]
The sequence $M_n^{(n)}-m_n$ is tight, by~\cite[Theorem~1.4]{Mallein2015}.
Therefore, it is sufficient to show~\eqref{eq:two_speed_non_lattice} and~\eqref{eq:two_speed_lattice} uniformly in $x \in K$, for~$K$ compact.
Let
\begin{equation*}
    T_n(A,B)\coloneq\{|u|=t_n:V^{(n)}(u)\in[-B\sqrt{n},-A\sqrt{n}]\},
\end{equation*}
and define, for  $|u|=t_n$, $M_{n-t_n}^{(u)}\coloneq\max_{v\succeq u,|v|=n}\{V^{(n)}(v)-V^{(n)}(u)\}$.
For $x \in K$, we have
\begin{equation}\label{eq:two_speed_bound_1}
    \P \left( \max_{u\in T_n(A,B)}\{V^{(n)}(u)+M_{n-t_n}^{(u)}\}> m_n+x \right) \le \P \left( M_n^{(n)}> m_n+x \right),
\end{equation}
and
\begin{multline}\label{eq:two_speed_bound_2}
    \P \left( M_n^{(n)}> m_n+x \right) \le \P \left( \max_{u\in T_n(A,B)}\{V^{(n)}(u)+M_{n-t_n}^{(u)}\}>m_n+x \right) \\
    + \P \left( \exists|u|=n,V^{(n)}(u_{t_n})\notin[-B\sqrt{n},-A\sqrt{n}],V^{(n)}(u)>m_n+\inf K \right),
\end{multline}
uniformly in $x \in K$.
Therefore, by Lemma~\ref{lem:two_speed_trajectory}, it is sufficient to compute the left-hand side of~\eqref{eq:two_speed_bound_1}.

Let
\begin{equation*}
    T_{1,n}(A,B)\coloneq\{|u|=t_n:V_1(u)\in[-B\sqrt{n},-A\sqrt{n}]\}.
\end{equation*}
By the Markov property at time~$t_n$, we have
\begin{align}
    \P \left( \max_{u \in T_n(A,B)} \{V^{(n)}(u)+M_{n-t_n}^{(u)}\} \le m_n+x \right)
    & = \E\left[\prod_{u\in T_n(A,B)}\mathds{1}_{\{V^{(n)}(u)+M_{n-t_n}^{(u)}\le m_n+x\}}\right] \\
    & = \E \left[ \exp \left( \sum_{u \in T_{1,n}(A,B)} \log \P_{V_1(u)}(M_{2,n-t_n} \le m_n+x) \right) \right]. \label{eq:two_speed_markov}
\end{align}
To compute $\P_{V_1(u)}(M_{2,n-t_n} > m_n+x)$, we distinguish the lattice and non-lattice cases.
\begin{enumerate}
    \item If $(V_2(u):u\in\T)$ is non-lattice, then, as $n \to \infty$,
    \begin{align}
        \P_y(M_{2,n-t_n} > m_n+x)
        & = \P \left( M_{2,n-t_n} > m_{n-t_n}-\frac{3}{2}\log\left(\frac{n}{n-t_n}\right)+x-y \right) \nonumber \\
        & \sim \frac{C_{2,\mathrm{NL}}}{(1-t)^{3/2}} (-y) \e^{-x+y-y^2/(2(n-t_n)\sigma_2^2)}, \label{eq:two_speed_equivalent_non_lattice}
    \end{align}
    uniformly in $y \in [-B\sqrt{n},-A\sqrt{n}]$ and $x \in K$, by Corollary~\ref{cor:uniform_convergence}.
    \item Now assume that $(V_1(u):u\in\T)$ is $(h_1,a_1)$-lattice, $(V_2(u):u\in\T)$ is $(h_2,a_2)$-lattice, and $h_1\Z \subset h_2\Z$.
    Note that $\{V_1(u):|u|=t_n\} \subset  [-B\sqrt{n},-A\sqrt{n}] \cap (a_1t_n+h_1\Z)$.
    Besides, the condition $h_1\Z \subset h_2\Z$ ensures that, for any $y \in [-B\sqrt{n},-A\sqrt{n}] \cap (a_1t_n+h_1\Z)$ and $x \in A_n \coloneq -m_n + a_1t_n + a_2(n-t_n) + h_2\Z$,
    \begin{equation*}
        -\frac{3}{2}\log\left(\frac{n}{n-t_n}\right)+x-y \in \{y : m_{n-t_n}+y \in a_2(n-t_n)+h_2\Z\}.
    \end{equation*}
    Therefore, Corollary~\ref{cor:uniform_convergence} yields that, as $n \to \infty$,
    \begin{equation}\label{eq:two_speed_equivalent_lattice}
        \P_y(M_{2,n-t_n} > m_n+x) \sim \frac{C_{2,\mathrm{L}}}{(1-t)^{3/2}} (-y) \e^{-x+y-y^2/(2(n-t_n)\sigma_2^2)},
    \end{equation}
    uniformly in $y \in [-B\sqrt{n},-A\sqrt{n}] \cap (a_1t_n+h_1\Z)$ and $x \in A_n \cap K$.
\end{enumerate}
In order to show~\eqref{eq:two_speed_non_lattice} and~\eqref{eq:two_speed_lattice} simultaneously, let us define $C_{2,\circ} = C_{2,\mathrm{NL}}$ and $A_n' = \R$ in the first case, $C_{2,\circ} = C_{2,\mathrm{L}}$ and $A_n' = A_n$ in the second case.
Inserting the equivalent~\eqref{eq:two_speed_equivalent_non_lattice} or~\eqref{eq:two_speed_equivalent_lattice}, according to the case we treat, into~\eqref{eq:two_speed_markov}, we obtain
\begin{equation*}
    \P \left( \max_{u \in T_n(A,B)} \{V^{(n)}(u)+M_{n-t_n}^{(u)}\} \le m_n+x \right) = \E \left[ \exp \left( (1+o_n(1)) \frac{C_{2,\circ} \e^{-x}}{(1-t)^{3/2}} Y_n(A,B) \right) \right],
\end{equation*}
where~$o_n(1)$ is deterministic and converges to~$0$ uniformly in $x \in A_n' \cap K$, and
\begin{equation*}
    Y_n(A,B) \coloneq \sum_{u\in T_{1,n}(A,B)} V_1(u) \e^{V_1(u)-V_1(u)^2/(2(n-t_n)\sigma_2^2)}.
\end{equation*}
By~\cite[Corollary~1.3]{Madaule2016},
\begin{equation*}
    Y_n(A,B) \xrightarrow[n \to \infty]{} Y_\infty(A,B) \coloneq -Z_{1,\infty}\sqrt{\frac{2}{\pi}}\int_{A/(\sigma_1\sqrt{t})}^{B/(\sigma_1\sqrt{t})}y^2\e^{-(\sigma_1^2t+\sigma_2^2(1-t))y^2/(2\sigma_2^2(1-t))} \d y, \quad \text{in probability}.
\end{equation*}
We have
\begin{multline*}
    \sup_{x \in A_n' \cap K} \left| \P \left( \max_{u \in T_n(A,B)} \{V^{(n)}(u)+M_{n-t_n}^{(u)}\} \le m_n+x \right) - \E \left[ \exp \left( \frac{C_{2,\circ} \e^{-x}}{(1-t)^{3/2}} Y_\infty(A,B) \right) \right] \right| \\
    \le \E \left[ \sup_{x \in A_n' \cap K} \left| \exp \left( (1+o_n(1)) \frac{C_{2,\circ} \e^{-x}}{(1-t)^{3/2}} Y_n(A,B) \right) - \exp \left( \frac{C_{2,\circ} \e^{-x}}{(1-t)^{3/2}} Y_\infty(A,B) \right) \right| \right].
\end{multline*}
Note that, in view of the definition of $Y_n(A,B)$ and $T_{1,n}(A,B)$, there exists $n_0 \ge 1$ such that, for any $n \ge n_0$, we have $Y_n(A,B) \le 0$ and then $Y_\infty(A,B) \le 0$. According to a basic inequality $|\e^{-x}-\e^{-y}|\le |x-y|\wedge1
$ for $x,y\ge 0$, we know that for $n$ large enough,
\begin{multline*}
\sup_{x \in A_n' \cap K} \left| \exp \left( (1+o_n(1)) \frac{C_{2,\circ} \e^{-x}}{(1-t)^{3/2}} Y_n(A,B) \right) - \exp \left( \frac{C_{2,\circ} \e^{-x}}{(1-t)^{3/2}} Y_\infty(A,B) \right) \right| \\
\le \left(\e^{\sup_{x\in K}|x|}\frac{C_{2,\circ}}{(1-t)^{3/2}}\left|(1+o_n(1))Y_n(A,B)-Y_\infty(A,B)\right|\right)\wedge 1.
\end{multline*}
Therefore, by the dominated convergence theorem and since~$K$ is compact, we deduce that
\[\E \left[ \sup_{x \in A_n' \cap K} \left| \exp \left( (1+o_n(1)) \frac{C_{2,\circ} \e^{-x}}{(1-t)^{3/2}} Y_n(A,B) \right) - \exp \left( \frac{C_{2,\circ} \e^{-x}}{(1-t)^{3/2}} Y_\infty(A,B) \right) \right| \right] \xrightarrow[n \to \infty]{} 0.\]
Combined with~\eqref{eq:two_speed_bound_1}, \eqref{eq:two_speed_bound_2}, and Lemma~\ref{lem:two_speed_trajectory}, letting $n\to\infty$ first, then $A\to0^+$ and $B\to\infty$, we conclude that 
\begin{equation*}
    \lim_{n \to \infty} \sup_{x \in A_n' \cap K} \left| \P \left( M_n \le m_n+x \right) - \E \left[ \exp \left( -\frac{C_{2,\circ} \e^{-x} Z_{1,\infty}}{(1-t+t\sigma_1^2/\sigma_2^2)^{3/2}} \right) \right] \right| = 0,
\end{equation*}
which completes the proof, thanks to tightness.
\end{proof}

\appendix

\section{Some analytic lemmas}

In this section, we consider some function $f \colon \N \times \R \to \R$ as well as two sequences $(\beta_n)$ and $(\gamma_n)$ such that $0 \le \beta_n = o(\gamma_n)$ and $\gamma_n \to \infty$ as $n \to \infty$.
The following analytic lemmas are useful throughout the paper (to derive Corollaries~\ref{cor:tail_distribution}, \ref{cor:uniform_convergence}, \ref{cor:simple_barrier_estimate}, \ref{cor:double_barrier_estimate}, \ref{cor:upper_bound_G}, and~\ref{cor:expectation_equivalent}), with appropriate choices of~$\beta_n$ and~$\gamma_n$.

\begin{lemma}\label{lem:from_sequence_to_large_constant}
    Assume that, for any sequence $(x_n)$ such that $x_n \to \infty$ and $x_n = O(\gamma_n)$, we have $f(n,x_n) \to 0$ as $n \to \infty$.
    Then, $\limsup_{x \to \infty} \limsup_{n \to \infty} |f(n,x)| = 0$.
\end{lemma}

\begin{proof}
    Assume that $\limsup_{x \to \infty} \limsup_{n \to \infty} |f(n,x)| > 0$.
    Then, there exist $\vep > 0$ and a sequence $(y_k)$ such that $y_k \to \infty$ as $k \to \infty$ and, for any~$k$, $\limsup_{n \to \infty} |f(n,y_k)| > \vep$.
    This implies that, for any~$k$, we can find~$n_k$ as large as we want so that $|f(n,y_k)| > \vep$.
    In particular, we can set $n_0 = 0$ and define inductively, for $k \ge 1$,
    \begin{equation*}
        n_k = \inf\{n > n_{k-1} : \gamma_n \ge y_k \ \text{and} \ |f(n,y_k)| > \vep\}.
    \end{equation*}
    We set $x_{n_k} = y_k$.
    Thus, we have $x_{n_k} \to \infty$, $x_{n_k} = O(\gamma_{n_k})$, and $f(n_k,x_{n_k}) \not\to 0$ as $k \to \infty$, which contradicts the assumption of the lemma.
\end{proof}

\begin{lemma}\label{lem:from_sequence_to_uniformity}
    Assume that, for any sequence $(x_n)$ such that $\beta_n \le x_n \le \gamma_n$, we have $f(n,x_n) \to 0$ as $n \to \infty$.
    Then, we have $\limsup_{n \to \infty} \sup_{x \in [\beta_n,\gamma_n]} |f(n,x)| = 0$.
\end{lemma}

\begin{proof}
    Similarly to Lemma~\ref{lem:from_sequence_to_large_constant}, if $\limsup_{n \to \infty} \sup_{x \in [\beta_n,\gamma_n]} |f(n,x)| > 0$, then there exist $\vep > 0$ and sequences $(n_k)$ and $(x_{n_k})$ such that $n_k \to \infty$, $\beta_{n_k} \le x_{n_k} \le \gamma_{n_k}$, and $|f(n_k,x_{n_k})| > \vep$, which contradicts the assumption of the lemma.
\end{proof}

\begin{lemma}\label{lem:combine_regimes}
    Assume that, for any sequence $(x_n)$ such that
    \begin{enumerate}
        \item either $\beta_n \le x_n = o(\gamma_n)$,
        \item or $x_n \asymp \gamma_n$,
    \end{enumerate}
    we have $f(n,x_n) \to 0$ as $n \to \infty$.
    Then the latter convergence holds for any choice of $(x_n)$ such that $\beta_n \le x_n = O(\gamma_n)$.
\end{lemma}

\begin{proof}
    Assume that there exists a sequence $(x_n)$ such that $\beta_n \le x_n = O(\gamma_n)$ and $f(n,x_n) \not\to 0$ as $n \to \infty$.
    Then, there exist $\vep > 0$ and a subsequence $(x_{n_k})$ such that $\liminf_{k \to \infty} |f(n_k,x_{n_k})| > \vep$.
    If $x_{n_k} = o(\gamma_{n_k})$, it directly contradicts the assumption of the lemma.
    Otherwise, there exists a subsequence $(x_{n_{k_j}})$ such that $x_{n_{k_j}} \asymp \gamma_{n_k}$.
    Since $\liminf_{j \to \infty} |f(n_{k_j},x_{n_{k_j}})| > \vep$, it still contradicts the assumption of the lemma.
\end{proof}

\section{Proof of Lemma~\ref{lem:ballot_additional_variables}}\label{sct:ballot_additional_variables}

We first prove the following lemma, which is a variation of~\cite[Lemma~C.1]{Aidekon2013}.

\begin{lemma}\label{lem:ballot_additional_variables_2}
    With the same notation as in Lemma~\ref{lem:ballot_additional_variables} and setting $\xi = \xi_1$, there exists $C = C(\sigma^2) > 0$ such that, for any $x > 0$,
    \begin{equation}\label{eq:ballot_additional_variables_2}
        \P(\exists k \le n, S_k > x+f_n(k)-\xi_k, \forall j \le n, S_j \le x+f_n(j)) \le C \frac{1+x}{\sqrt{n}} \left( \P(\xi \ge 0)^{1/2} + \E \left[ \xi_+^2 \right]^{1/2} \right).
    \end{equation}
\end{lemma}

Our argument corrects a mistake from~\cite[Lemma~A.1]{Mallein2016}.
The problem might come from an incorrect expansion of the expression $\mathds{1}_{\{\xi \ge 0\}} (1+X_++\xi_+)(1+\xi_+)$, that appears below.

\begin{proof}
    By applying successively the union bound, the Markov property, and Lemma~\ref{lem:mallein_random_walk_estimate}, we bound the left-hand side in~\eqref{eq:ballot_additional_variables_2} by
    \begin{align*}
        & \sum_{k=0}^{n-1} \P(S_k > x+f_n(k)-\xi_k, \forall j \le n, S_j \le x+f_n(j)) \\
        & \quad = \sum_{k=0}^{n-1} \E \left[ \mathds{1}_{\{S_k > x+f_n(k)-\xi_k, \forall j \le k, S_j \le x+f_n(j))\}} \P_{S_k}(\forall j \le n-k, S_j \le f_n(k+j)) \right] \\
        & \quad \le C \sum_{k=0}^{n-1} \E \left[ \mathds{1}_{\{S_k > x+f_n(k)-\xi_k, \forall j \le k, S_j \le x+f_n(j))\}} \frac{1+(x+f_n(k)-S_k)}{\sqrt{n-k}} \right].
    \end{align*}
    Note that, on the event where $S_k > x+f_n(k)-\xi_k$, we have $x+f_n(k)-S_k \le \xi_k$.
    Applying this bound, then conditioning on $(S_k-S_{k-1}, \xi_k)$ and applying Lemma~\ref{lem:mallein_random_walk_estimate}, we obtain that the left-hand side in~\eqref{eq:ballot_additional_variables_2} is bounded by
    \begin{equation*}
        C \sum_{k=0}^{n-1} \E \left[ \mathds{1}_{\{\xi \ge 0\}} \frac{(1+x)(1+(X+\xi)_+ \wedge \sqrt{k})(1+\xi \wedge \sqrt{k})}{(k+1)^{3/2}} \frac{1+\xi}{\sqrt{n-k}} \right],
    \end{equation*}
    where $(X,\xi) = (S_1,\xi_1)$.
    Observe that, for $a, b \ge 0$,
    \begin{align*}
        \sum_{k=0}^{n/2} \frac{(1 + a \wedge \sqrt{k})(1+ b \wedge \sqrt{k})}{(k+1)^{3/2}}
        & \le \sum_{k=0}^\infty \frac{(1 + a \wedge \sqrt{k} + b \wedge \sqrt{k} + (ab) \wedge k)}{(k+1)^{3/2}} \\
        & \le C(1 + \log_+ a + \log_+ b + \sqrt{ab}) \\
        & \le C(1+a+b).
    \end{align*}
    Therefore, up to a multiplicative constant, the left-hand side in~\eqref{eq:ballot_additional_variables_2} is bounded by
    \begin{equation*}
        \frac{1+x}{\sqrt{n}} \E \left[ \mathds{1}_{\{\xi \ge 0\}} (1+X_++\xi_+)(1+\xi_+) \right] \le C \frac{1+x}{\sqrt{n}} \left( \P(\xi \ge 0) + \P(\xi \ge 0)^{1/2} + \E \left[ \xi_+^2 \right] + \E \left[ \xi_+^2 \right]^{1/2} \right),
    \end{equation*}
    by Cauchy--Schwarz inequality, since~$X$ admits a finite second moment.
    Note that $\P(\xi \ge 0) \le \P(\xi \ge 0)^{1/2}$.
    Moreover, if $\E \left[ \xi_+^2 \right] \le 1$, then $\E \left[ \xi_+^2 \right] \le \E \left[ \xi_+^2 \right]^{1/2}$, and otherwise, we can use Lemma~\ref{lem:mallein_random_walk_estimate} to bound
    \begin{align*}
        \P(\exists k \le n, S_k > x+f_n(k)-\xi_k, \forall j \le n, S_j \le x+f_n(j))
        & \le \P(\forall j \le n, S_j \le x+f_n(j)) \\
        & \le C \frac{1+x}{\sqrt{n}} \\
        & \le C \frac{1+x}{\sqrt{n}} \left( \P(\xi \ge 0)^{1/2} + \E \left[ \xi_+^2 \right]^{1/2} \right).
    \end{align*}
    In both cases, the bound~\eqref{eq:ballot_additional_variables_2} holds.
\end{proof}

We now show Lemma~\ref{lem:ballot_additional_variables} by adapting the argument of~\cite[Lemma~3.3]{Mallein2016}.
As explained above, mind that there might be some missing terms in the bound of the latter.

\begin{proof}[Proof of Lemma~\ref{lem:ballot_additional_variables}]
    Define $\tau_n(x) = \inf\{k < n : S_k > x + f_n(k) - \xi_{k+1}\}$.
    We split the left-hand side of~\eqref{eq:ballot_additional_variables} into two, depending on whether the hitting time~$\tau_n(x)$ occurs before or after time~$\lfloor n/2 \rfloor$.
    To simplify notations, assume that~$n$ is even.
    This leads to the following first term,
    \begin{multline}\label{eq:ballot_additional_variables_1}
        \P(\forall k \le n, S_k \le x+f_n(k), \tau_n(x) < n/2, x+f_n(n)-S_n \in [y,y+1]) \\
        = \E \left[ \mathds{1}_{\{\forall k \le n/2, S_k \le x+f_n(k), \tau_n(x) < n/2\}} \P_{S_{n/2}}(\forall k \le n/2, S_k \le x+f_n(k+n/2), x+f_n(n)-S_{n/2} \in [y,y+1]) \right].
    \end{multline}
    Now, applying the Markov property and Lemma~\ref{lem:mallein_random_walk_estimate}, we obtain that the last probability is bounded by $C(1+y)/n$.
    Now, note that
    \begin{equation*}
        \{\tau_n(x) < n/2\} \subset \{\exists j \in (0,n/2], S_j > x+f_n(j-1)-(\xi_j-S_j+S_{j-1})\},
    \end{equation*}
    where $(S_j-S_{j-1}, \xi_j-S_j+S_{j-1})$, for $j \ge 1$, are \iid random variables on~$\R^2$.
    Therefore, by Lemma~\ref{lem:ballot_additional_variables_2},
    \begin{equation*}
        \P(\forall k \le n/2, S_k \le x+f_n(k), \tau_n(x) < n/2) \le C \frac{1+x}{\sqrt{n}} \left( \P(\xi_1 - S_1 \ge 0)^{1/2} + \E \left[ (\xi_1 - S_1)_+^2 \right]^{1/2} \right).
    \end{equation*}
    Hence, the first term~\eqref{eq:ballot_additional_variables_1} is bounded by the right-hand side of~\eqref{eq:ballot_additional_variables}.

    We treat the second term by time reversal.
    Define $\hat{S}_k=S_n-S_{n-k}$, $\hat{\xi}_k=\xi_{n-k+1}+1$, and $\hat{f}_n(k)=f_n(n)-f_n(n-k)$.
    Then, the random variables $(\hat{S}_k-\hat{S}_{k-1},\hat{\xi}_k)$ are \iid with the same law as $(S_1,\xi_1+1)$.
    We have
    \begin{align*}
        & \P(\forall k \le n, S_k \le x+f_n(k), \tau_n(x) \in[n/2,n), x+f_n(n)-S_n \in [y,y+1]) \\
        & \quad \le \P(\forall k \le n, \hat{S}_k \ge \hat{f}_n(k)-(y+1), \exists j \le n/2, \hat{S}_j < \hat{f}_n(j)-(y+1)+\hat{\xi}_j, \hat{S}_n-\hat{f}_n(n)+(y+1) \in [x,x+1]) \\
        & \quad \le C \frac{(1+x)(1+y)}{n^{3/2}} \left( \P(\xi_1+1 \ge 0)^{1/2} + \E \left[ (\xi_1+1)_+^2 \right]^{1/2} \right),
    \end{align*}
    by the Markov property and Lemmas~\ref{lem:mallein_random_walk_estimate} and~\ref{lem:ballot_additional_variables_2}.
    This concludes the proof.
\end{proof}

\section*{Acknowledgements}

We warmly thank Bastien Mallein and Michel Pain for their help, insightful advice, and careful reading.
We also acknowledge partial support from the MITI interdisciplinary program 80PRIME GEx-MBB, the ANR MBAP-P (ANR-24-CE40-1833) project, and the MINT fellowship program.
Furthermore, Louis Chataignier is supported by a French ministerial doctoral fellowship (MESRI) and Lianghui Luo is supported by China Scholarship Council (No.202306040031).

\bibliographystyle{abbrv}
\bibliography{biblio}

\end{document}